\documentclass[11pt]{article}

\usepackage{amssymb}
\usepackage{amsmath}
\usepackage{amsthm}
\usepackage[english]{babel}
\usepackage[utf8]{inputenc}
\usepackage{latexsym}
\usepackage{mathrsfs}
\usepackage{bbm}
\usepackage{graphicx}
\usepackage[hidelinks]{hyperref}
\usepackage{accents}
\usepackage{cases}
\usepackage{eurosym}
\usepackage{url}
\usepackage{color}
 
\newcommand{\R}{\mathbb{R}}

\newcommand{\E}{\mathcal{E}}
\newcommand{\F}{\mathcal{F}}

\DeclareMathOperator{\esssup}{esssup}
\def\esssup_#1{\underset{#1}{\mathrm{ess\,sup\, }}}
\def\essinf_#1{\underset{#1}{\mathrm{ess\,inf\, }}}
\def\argmax_#1{\underset{#1}{\mathrm{arg\,max\, }}}
\def\argmin_#1{\underset{#1}{\mathrm{arg\,min\, }}}

\def \ep{\hbox{ }\hfill$\Box$}

\def\b1{\bf 1}

\def \R{\mathbb{R}}

\def \E{\mathbb{E}}
\def \F{\mathbb{F}}

\def \P{\mathbb{P}}

\def \Ac{{\cal A}}

\def \Dc{{\cal D}}

\def \Fc{{\cal F}}

\def \Vc{{\cal V}}

\def \Vc{{\cal V}}

\def \th{\theta}

\def \ep{\hbox{ }\hfill$\Box$}

\newtheorem{Theorem}{Theorem}[part]
\newtheorem{Definition}{Definition}[part]
\newtheorem{Proposition}{Proposition}[part]

\newtheorem{Lemma}{Lemma}[part]

\newtheorem{Remark}{Remark}[part]

\def\reff#1{{\rm(\ref{#1})}}
\def\beqs{\begin{eqnarray*}}
\def\enqs{\end{eqnarray*}}
\def\beq{\begin{eqnarray}}
\def\enq{\end{eqnarray}}

\newcommand{\rr}{\mathbb{R}}

\newcommand{\eee}{\mathbb{E}}

\newcommand{\aaa}{\mathcal{A}}

\newcommand{\h}{\alpha}

\newcommand{\pz}{\pi_0}
\newcommand{\po}{\pi_1}

\begin{document}

\title{A McKean-Vlasov approach to distributed electricity generation development
\footnote{This work is supported by FiME (Finance for Energy Market Research Centre) and the ``Finance et D\'eveloppement Durable - Approches Quantitatives'' EDF - CACIB Chair.}
}

\author{
Ren\'e A\"id
\footnote{LEDa, Université Paris Dauphine, PSL research university, \sf rene.aid at dauphine.fr}
\qquad
Matteo Basei
\footnote{IEOR, University of California, Berkeley, and EDF R\&D \sf matteo.basei at edf.fr}
\qquad
Huy\^en Pham
\footnote{LPSM, Université Paris Diderot and CREST-ENSAE, \sf pham at lpsm.paris}
}

\maketitle

\begin{abstract} 
This paper analyses the interaction between centralised carbon emissive technologies and distributed intermittent non-emissive technologies. In our model, there is a representative consumer who  can satisfy her electricity demand by investing in distributed generation (solar panels) and by buying power from a centralised firm at a price the firm sets. Distributed generation is intermittent and induces an externality cost to the consumer. The firm provides non-random electricity generation subject to a carbon tax and to transmission costs. The objective of the consumer is to satisfy her demand while mini\-mising investment costs, payments to the firm and intermittency costs. The objective of the firm is to satisfy the consumer's residual demand while minimising investment costs, demand deviation costs, and maximising the payments from the consumer. We formulate the investment decisions as McKean-Vlasov control problems with stochastic coefficients. We provide explicit, price model-free solutions to the optimal decision problems faced by each player, the solution of the Pareto optimum, and the Stackelberg  equilibrium where the firm is the leader.  We find that, from the social planner's point of view, the carbon tax or transmission costs are necessary to justify a positive share of distributed capacity in the long-term, whatever the respective investment costs of both technologies are. The Stackelberg equilibrium is far from the Pareto equilibrium and leads to an over-investment in distributed energy and to a much higher price for centralised energy.  
\end{abstract}

\vspace{3mm}
\noindent {\bf JEL Classification}:  L94, C73, C61, O33, Q41, Q42.

\vspace{3mm}
\noindent {\bf MSC Classification}: 91B42,  93E20, 91A15.

\vspace{3mm}
\noindent {\bf Key words}: decarbonation, distributed generation, stochastic game,  McKean-Vlasov. 

\section{Introduction}

This paper analyses the interaction between centralised carbon emissive technologies and distributed intermittent non-emissive technologies. We start with the premise that distributed energy sources (mainly solar energy) give consumers an instrument that increases their bargaining power in their interaction with electric firms. This new deal in power systems entails a strategic interaction between these firms industry-wide and consumers now empowered with generation capacities. But this interaction is not of purely competitive type. Distributed generation is intermittent, highly dependent on weather conditions, and might not be enough at each instant to satisfy consumers' electricity demand. To avoid extra costs to compensate for the intermittency of distributed generation (e.g.~batteries), consumers might still want to have access to centralised generation as insurance. 

The economic literature on the development of renewable energy has focused on the real values of renewable energy technologies (wind and solar) and in particular, on the externality cost they induce in electric systems, and on what should be the optimal share of renewable energy in a given power system. The main references on this thoroughly studied subject are Joskow (2011) \cite{Joskow11}, Borenstein (2012) \cite{Borenstein12}, Hirth (2013) \cite{Hirth13} and more recently by Gowrisankaran et al.~(2016) \cite{Gowrisankaran16}. Another topic is the existence or not of a capacity threshold  where renewable energy is viable and no longer needs subsidies. Green and Léautier (2015) \cite{Green15} address this question and find that achieving a low-carbon power system without permanent subsidy might not be possible because renewable energy drives electricity prices down as its market share increases. Another important issue is how the distribution tariffs should be adapted to take into account the new phenomenon of power injection at the distribution level (see \cite{Brown17} and references therein for a description of the metering problem). To our knowledge, our paper is the first to study the development of renewable energy as a strategic interaction situation between consumers and electricity firms.

We consider a representative consumer and a firm representing the electricity producers as a whole. Since our focus is on the effect of the distributed generation intermittency, we assume that the consumer's demand is inelastic and constant. The consumer can invest in solar energy while the firm invests in power plants. Both perform their investments at a continuously controlled investment rate. We assume that the firm invests in dispatchable generation facilities (i.e.~non-random) and sells its production to the consumer at a price he fixed. The dispatchable technology is carbon emissive and its generation is subject to a carbon tax. The carbon tax decreases as the distributed non-emissive technology share increases. The objective of the firm is to satisfy the electricity demand of the consumer at least total discounted expected costs. The firm's cost function has three parts: the investment cost, the penalty for deviating from demand, and the revenue from selling electricity to the consumer.  The firm decides the price at which to sell its energy to the consumer. This price is a random process. The consumer's objective is also to minimise the total discounted expected costs to satisfy her demand, either by using her own random solar energy or by buying energy from the firm. The consumer has disutility from the intermittency of  distributed energy. This disutility  is quantified as a penalisation of the variance of the distributed energy. Further, the centralised energy produced by the firm is characterised by the existence of transmission costs that are assumed to be proportional to the energy price and that are paid by the consumer.  Although the burden of transmission costs can be allocated by the transmission operator to the firm in some countries (e.g.~in the United Kingdom), it is transferred to the consumer by the firm.

Although this model simplifies the complexities of the investment decisions in electricity generation, it captures a key new feature in today's power systems and a new economic problem facing electric firms. If the firm decides to charge a high price for its energy to maximise its profit, then the consumer has the possibility to escape the firm's market power by investing in distributed energy and generating her own electricity.  However, charging too low a price to avoid this situation could impede the firm from recovering its investment costs. This trade-off is at the core of the interaction problem between the consumer and the firm. If the consumer has no other alternative than buying electricity from the firm, the equilibrium will be fixed by the standard theory of marginal cost pricing that depends on the demand elasticity of the consumer (see Boiteux (1956,1960) \cite{Boiteux56,Boiteux60}). Here, the availability of the option to invest in distributed generation changes the respective bargaining powers of the consumer and the firm. The consumer can escape the firm's market power at the expense of investing in a costly alternative. Thus, the firm might be tempted to offer suboptimal prices to the consumer to prevent her from investing in local generation.

This paper's first contribution is to give the optimal investment strategy for both players without any assumptions in the model on the energy price. The optimal investment strategies are solutions of stochastic control problems with random coefficients. Moreover, in the case of the consumer's problem, the penalisation of the variance leads to a new type of stochastic control problem, called  McKean-Vlasov control problems. Both problems admit closed-form solutions depending on the price process of the centralised energy. Using only the hypothesis that the price process admits a long-term stationary value, we derive the optimal long-term behaviour of each player. In the long-term, the consumer's cumulated distributed capacity will reach a positive limit, as long as the long-term price of centralised energy plus transmission cost  is greater than the annuity of distributed energy (total discounted investment costs). In particular, the consumer decides her long-term cumulated capacity based on the centralised energy price compared to the distributed energy cost and her aversion to intermittency. Centralised energy costs which are too high can result in investment that exceeds consumers' demand. The firm's cumulated capacity will be equal to the residual demand of the consumer (demand minus installed distributed capacity). The independence of these results with  respect to the price models guarantees their robustness. 

Our second contribution is to provide the conditions for the existence and characterisation of a long-term Pareto optimum between the distributed and centralised generations. We compute this optimum and give conditions under which it can be implemented by  setting a regulated price for centralised energy.  Once again, we obtain these results by considering the centralised energy price as a random parameter, that is without making any model hypothesis. 

We find that if the price of the centralised energy admits a stationary distribution, then the levels of installed centralised and distributed generation determined by the social planner admits stationary limits. The long-run level of centralised generation is equal to the residual demand of the consumer while the level  of the distributed generation nontrivially depends  on the costs of both the distributed and centralised technology costs. However, some limiting cases are worth noting. In the limiting case where there are transmission costs and no carbon tax, distributed generation admits a long-term positive optimum value. This value increases with the long-run value of the centralised energy price, decreases with the distributed investment cost annuity and with the intermittency measured by the variance of the distributed generation. This long-run optimum is achieved at an exponential rate. When there is neither a carbon tax nor transmission costs, there is no investment in distributed generation in the long-run, whatever the respective investment costs of both technologies: The social planner who is risk-averse to intermittency will invest in high cost dispatchable technology rather than in low cost intermittent technology.  This unexpected result is not a consequence of the variance penalisation, but is actually induced by  the (non-random) demand satisfaction objective.  Indeed, the intermittency of distributed capacity always results in a positive cost whereas dispatchable centralised technology can cancel this cost by  perfectly matching the demand.  
 
Moreover, we can implement a stationary Pareto optimum price. By fixing the long-run value of the centralised energy price, the social planner can achieve the Pareto optimal long-run value of distributed generation, under certain explicit conditions. This asymptotic Pareto optimum price  can be interpreted as a long-term contract offered to the firm to guarantee its return on investments. We find that as soon as the distributed energy per unit cost has reached a certain level, the free  interaction between the technologies will lead to an equilibrium in the long-run. Thus, for a given target  level of distributed energy, there is a distributed energy cost under which subsidies are no longer required to achieve an asymptotic Pareto optimum. This result somehow tempers the above cited conclusion of Green and Léautier (2015) \cite{Green15}. We extend our model to the situation of multiple consumers to check whether or not the results depend on the reduction of the interaction between technologies to a single representative agent. We find that the same general results still apply, up to a correcting term given by the return to scale provided by the cost reduction induced with a large number of players. Moreover, some important limiting cases admit further analysis. In particular, we find that the implementation of the asymptotic Pareto optimum relies on the existence of a positive carbon tax. Contrary to the case where the social planner could enforce the long-run equilibrium by choosing himself the level of investment, the use of a price requires a positive carbon tax to ensure a positive distributed generation. In this situation, the transmission costs are no longer enough to sustain the development of distributed energy sources on the long-run.

Our third contribution is the study of a the Stackelberg equilibrium where the firm is the leader. This modelling represents an unregulated market interaction between the two players. In this game, the firm chooses both its installed capacity and its selling price and then, the consumer chooses her level of installed solar panels. Because the firm determines its price and the consumer's demand is inelastic, a Nash equilibrium is not a realistic consideration while the Stackelberg equilibrium captures the asymmetry of the two players. We provide an explicit condition under which there is an asymptotic Stackelberg equilibrium that we compute. We find that it leads to significant differences than in the Pareto optimum such as a higher price for the centralised energy and a significant over-investment in distributed energy.
 
As a result, this comparative analysis of different interactions between the firm and the consumer shows that if there is not a large carbon emission price, it is unlikely that we can achieve a social welfare Pareto optimum either by letting the players freely interact or by a public intervention limited to a regulated centralised energy price. 

The outline of the paper is as follows. Section~\ref{sec:Model} presents our model and its hypotheses. In Section~\ref{sec:Players} we describe the optimal behaviour of each player (firm, consumer, social planner). Section~\ref{sec:Equilibres} provides the different possible equilibria. Section~\ref{sec:Extensions} provides extensions with a large number of firms and consumers, and we conclude in Section \ref{sec:conclu}.

\section{Model and problem formulation}
\label{sec:Model}
\setcounter{equation}{0}
\setcounter{Assumption}{0} \setcounter{Theorem}{0}
\setcounter{Proposition}{0} \setcounter{Corollary}{0}
\setcounter{Lemma}{0} \setcounter{Definition}{0}
\setcounter{Remark}{0}

We consider an interaction model on electricity market  that involves the following actors: there is a firm (he) which provides electricity with centralised fully dispatchable generation (typically coal- or gas-fired plants) to a representative consumer who can also produce her own  electricity by investing in renewable energy, typically by purchasing solar panels.  The consumer's demand for electricity is then satisfied both with the distributed generation of the solar panels and with the centralised generation bought at the price set by the firm. The problem for the (representative) consumer is then to decide how many solar panels she must install. On the other hand, the firm has to satisfy the residual demand of the consumer with its generation plants while minimising his production costs and taking into account carbon taxes due to emission of thermal energy. Further, we consider a third player in this game, namely a social planner, who is typically looking for a price of the centralised generation, which is Pareto-optimal for the firm and the consumer. 

We now formulate mathematically the model and the optimisation problem for each actor in this three-player game. We fix some probability space $(\Omega,\Fc,\P)$ equipped with a filtration $\F$ $=$ $(\Fc_t)_{t\geq 0}$ satisfying the usual conditions and supporting  two independent standard Brownian motions $W$ and $W^0$. The consumer's cumulative net available capacity of distributed generation at any time $t$ is denoted by $X_t^\alpha$ and is governed by:
\beq \label{dynX}
dX_t^\alpha &=& b \alpha_t dt + \sigma X_t^\alpha dW_t, \qquad t \geq 0, \qquad X_0^\alpha = x_0 \in \R, 
\enq
where the $\F$-adapted control process $\alpha$ $=$ $(\alpha_t)_{t\geq 0}$ represents the solar panel installation rate ($\alpha_t$ $>$ $0$  means that the consumer buys the at time $t$, while $\alpha_t$ $<$ $0$ means that the consumer sells them), $b > 0$ is the load factor of the solar panels. Typically, the installation rate $\alpha_t$ is measured in megawatt per unit of time, and $b$ is a dimensionless parameter.   
The parameter $\sigma$ $>$ $0$ is a volatility coefficient related to the uncertainty in the net electricity available capacity of photovoltaics (its net capacity depends on weather conditions), and the proportional term w.r.t.~$X_t^\alpha$ in the noise means that the uncertainty increases with the level of generation. {The volatility term of the intermittent generation technology has been chosen to reflect that the more the consumer relies on this technology, the more he has to cope with uncertainty. Our model is not a geographical model where the total generation of intermittent electricity sources could result in a smoother generation by an aggregation effect. The form of the volatility captures the idea that the consumer faces a risk of a low production whatever the level of installed capacity. This point is a concern of many transmission system operators managing electric network with a high level of renewable energy penetration (see German Energy Regulator Annual Report on the constant increase of redispatching cost to cope with intermittent energy sources \cite{Bunde18}, Fig.~64, p. 173, and Fig.~65, p.~174).} The demand process of the (representative) consumer, denoted by $D$, is an exogenous factor in our model, and is assumed to be constant. We justify this hypothesis in the following way: The problem of the consumer as well as the social planner is to make an arbitrage between non-emissive intermittent technologies and emissive dispatchable technologies. Uncertainty on the demand only complicates the computations without bringing more insights on the choices of the consumer or the social planner. Another way of justifying this hypothesis is to consider that the volatility $\sigma$ is the net uncertainty generated by the distributed generation plus demand. In order to satisfy the consumer's residual demand (i.e.~the electricity demand not fulfilled by the photovoltaic) at any time $t$, the consumer purchases electricity from the firm at price $P_t$ as chosen by the firm. The price process $(P_t)_{t\geq 0}$ should not be considered as a real-time pricing process, typically the day-ahead hourly spot price that is quoted on electricity markets. It is more pertinent to see this price as the annual tariff that appears on consumers' electricity bill because the price is announced in advanced to the consumer by the firm in the case of a Stackelberg equilibrium. For this reason, the price process $P$ $=$ $(P_t)_{t\geq 0}$ is  adapted with respect to the filtration $\F^0$ $=$ $(\Fc_t^0)_{t\geq 0}$ generated by $W^0$. In other words, we assume that the price of the centralised generation is not affected by local issues like the weather, but only by macroeconomic and market factors. However, we do not make {\em a priori} model assumptions on the price process since our ultimate goal is to endogenously derive the price by equilibrium arguments. 
Over an infinite horizon, the expected total discounted cost of the consumer for an installation rate  $\alpha$ is: 
\beq \label{totalcost}
J_1(\alpha) &= & 
\E \Big[ \int_0^\infty e^{-\rho t} \Big( C_1(\alpha_t) + (1+\th) P_t (D - X_t^\alpha) +  \eta {\rm Var}[X_t^\alpha] \Big) dt \Big],  
\enq
where $\rho$ $>$ $0$ is some discount factor which we assume is  large enough for ensuring the well-posedness of our problem. The first term $C_1(\alpha_t)$ represents the instantaneous installation cost for photovoltaics where $C_1$ a convex function on $\R$, $C_1(0)$ $=$ $0$ which we assume in the sequel to be in the quadratic form:
\beqs \label{quadracout}
C_1(a) &=& c a + \gamma a^2, 
\enqs
where $c$ $\geq$ $0$ is the per unit constant cost of distributed power, and $\gamma$ $>$ $0$ is the market impact factor. 

The hypothesis of a constant cost $c$ of distributed energy is challenging because in real life, an important decrease in per unit cost of renewable energy sources in the last 20~years has occurred because of both economies of scale, learning by doing and technological progress (see Reichelstein et al.~\cite{Reichelstein13} and Rubin et al.~\cite{Rubin15}). Nevertheless, the study of the equilibrium between centralised and distributed generation when consumer has access to a technology at constant per unit cost $c$ already reveals a lot about their interaction. Furthermore, we present in Section \ref{sec:Extensions} an extension where we relax the hypothesis of a constant per unit cost of solar panels to take into account an increasing return to scale. 
The quadratic term of the cost function captures the fact that the marginal cost of 1 MW of solar panels will increase if the buying rate is too fast. If the representative consumer wants to increase her installation rate, solar panels production firms need to buy resources at higher prices and thus increase their price. This phenomenon was observed during the renewable energy boom (see \cite{Wiser11} p. 51, Fig. 32). Further, the small negative $a$ corresponds to a situation where the consumer sells back her solar panels. In this sense, the investment problem for the consumer has a reversibility property. The second term $(1+\theta) P_t (D - X_t^\alpha)$ is the (algebraic) cost for satisfying the residual demand with the centralised electricity. The parameter $\theta$ represents the costs that add up to the energy price $P_t$ on the consumer's electricity bill due to the necessary transmission and distribution infrastructure required to bring centralised generation to her household. In this regard, $\theta$ measures the extra cost induced by centralised generation compared to distributed generation. It is not small. Transmission costs  can represent as much as two-thirds of the consumer's bill, leading to an $\theta =2$. In this model, we assume that $\theta$ is exogenous and fixed and leave for future research the potential strategic effect it might generate between distributed energy adopters and non adopters (see \cite{Brown17} for a description of the effect of metering on the repartition of network costs between different type of consumers). The distributed generation might potentially exceed the demand which would correspond to a situation where the consumer sells the excess generation to the firm. But since our model is closed (there is no other consumer to whom this excess generation could be sold), we place conditions on the definition of an equilibrium that excludes this case {(for this reason, we keep the cost functions unchanged when the distributed generation exceeds the demand)}. The third term represents the preference of the consumer for low-variance policies (quantified by the penalty parameter $\eta$ $>$ $0$) so as to get stable flux in the distributed generation, and thus  be less sensitive  to the market risk of a centralised electricity price. The consumer has to decide how much distributed generation she has to self-produce from the solar panels, and her goal is to minimise over the installation rate control  $\alpha$ the expected total discounted cost $J_1(\alpha)$  in \reff{totalcost}. The solution to this problem will be studied in Section \ref{secconsu}. Notice that the corresponding optimal self-production $\hat X$ $=$ $\hat X(P)$ for the consumer depends on the thermal price $P$. 

On the other hand, knowing  the (optimal) distributed generation of the consumer, the firm has to choose its centralised generation capacity for satisfying the (residual) demand of the consumer while taking into account the operational costs, the gain from the sale of electricity to the consumer, and carbon taxes for carbon emission.  We denote $Q^\nu$ as the firm's cumulative production of centralised generation given by the dynamics
\beq
\label{dynQ}
dQ_t^\nu &=& \nu_t dt, \qquad t \geq 0, \qquad  Q_0^\nu = q_0 \in \R, 
\enq
where $\nu$ $=$ $(\nu_t)_{t\geq 0}$ is an $\F$-adapted real-valued process, which represents the production installation rate. However, compared to the generation from distributed generation,  there is no uncertainty in centralised generation  (no stochastic noise on the evolution of $Q^\nu$) which captures the fact that centralised generation is fully dispatchable. 
Given a level of distributed generation capacity $X^\alpha$, the aim of the energy firm is to minimise over its production rate $\nu$ of centralised generation capacity the expected total discounted cost:
\beq
J_2(\nu;X^\alpha) &=& \E \Big[ \int_0^\infty e^{-\rho t} \Big( C_2(\nu_t)  + \lambda\big(D - X^\alpha_t-Q_t^\nu\big)^2 \nonumber \\
& & \;\;\;\;\; - \;  P_t (D - X_t^\alpha)   +    \pi_t^{\alpha,\nu} (D - X_t^\alpha)    \Big) dt \Big].  \label{costenergy}
\enq
The first term in \reff{costenergy} represents the production cost, and therefore we consider a quadratic cost function of
\beqs
C_2(u) &=& h u^2, 
\enqs
for some positive constant $h$ $>$ $0$. Contrary to the investment cost function of the consumer, there is no linear term in the firm's investment cost function. We do not add this linear term because the investment in centralised generation capacity is irreversible (at least, much more irreversible than investing in distributed generation). Increasing or decreasing capacity induces important costs to the firm.  The second term  is a quadratic penalisation, which constrains (more and less depending on the size of the parameter $\lambda > 0$) the firm to fit the residual demand of the consumer with its generation. {In electric systems, generation is always equal to consumption. Thus, the commitment of the producer to satisfy the consumer's electricity demand translates into an almost sure constraint stating $Q^\nu_t = D - X^\alpha_t$. But, because $Q^\nu$ and $X^\alpha$ are controlled processes, it is not possible to ensure this equality almost surely. A natural way to circumvent this difficult in a highly stylised model of power system is to resort to a penalisation of the demand satisfaction constraint. The term $\lambda (D - X^\alpha_t - Q^\nu_t)^2$ can be interpreted as a penalisation of the demand satisfaction constraint in an augmented Lagrangian, a technique which is common in unit commitment model to find an optimal dispatch \cite{Renaud93}.}

The third term is the gain from the sale of centralised generation to the consumer at price $P_t$. Although the consumer pays the centralised energy $(1+\th) P_t$, the firm only receives $P_t$. The fourth term represents a carbon tax for the emission of  thermal energy, where we set
\beqs
\pi^{\alpha,\nu}_t &=& \frac{X^\alpha_t}{D} \pi_0 + \frac{Q^\nu_t}{D} \pi_1, \;\;\; t \geq 0, 
\enqs
for some constants $0$ $\leq$ $\pi_0$ $<$ $\pi_1$, as the amount of carbon tax to be paid for one unit of energy sold at time $t$. 
Since the demand is expected to fit the total production from distributed and centralised generations, $D$ $\simeq$ $X_t^\alpha + Q_t^\nu$, then  $\pi_t^{\alpha,\nu}$ is approximately a convex combination of $\pi_0$ and $\pi_1$ and is valued in  $[\pi_0,\pi_1]$. It increases with the level of emissive centralised generation, is close to $\pi_1$ when the demand is mostly satisfied by thermal production, and is close to $\pi_0$ when the distributed generation covers almost all the demand. We choose the parameter $\pi_0$ close to the smallest value of observed carbon emission allowances price ($\pi_0 \approx 2$~euro per metric ton) whereas we choose $\pi_1$ close to the penalty value of firm's missing their emission target  ($\pi_1 \approx 100$ per metric ton). Making both $\pi_0$ and $\pi_1$ equal to zero allows to study the case of an equilibrium with centralised non-emissive technology. The solution to this problem will be studied in Section \ref{seccomp}. Note that the corresponding optimal centralised capacity $\hat Q$ $=$ $\hat Q(X^\alpha)$ for the firm depends on the distributed generation $X^\alpha$ of the consumer. 
  
Besides the consumer and the centralised firm, we consider a {\it social planner}, whose aim is to minimise over the investment rate pair $(\alpha,\nu)$  the sum of expected costs from  the distributed and centralised generations:  
\beqs
J(\alpha,\nu) &=& J_1(\alpha) + J_2(\nu;X^\alpha) \\
&=& \E \Big[ \int_0^\infty e^{-\rho t} \Big( C_1(\alpha_t) + C_2(\nu_t) + (\th P_t + \pi_t^{\alpha,\nu}) (D- X_t^\alpha) \\
& & \;\;\;\;\;  + \;   \eta {\rm Var}[X_t^\alpha] +  \lambda\big(D - X^\alpha_t-Q_t^\nu\big)^2  \Big) \Big].  
\enqs
This problem will be studied in Section \ref{secsocial}. The corresponding optimal distributed and centralised generation capacities are denoted by  $(X^*,Q^*)$ $=$ $(X^*(P),Q^*(P))$. The cost function of the social planner still depends on the centralised energy price $P$. The energy price is not a pure transfer price from the consumer to the firm. It entails a social cost due to the need for the transmission infrastructure. The chief goal of the social planner is to endogenise the electricity market price $P$ by equilibrium arguments. We first consider the concept of  Pareto-optimality for the consumer and the energy firm, which is such that the optimal distributed generation capacity for the consumer and the optimal centralised generation capacity for the firm coincides with the solution derived from the social planner's problem. In other words, the mathematical problem is to determine a price process $P^*$ such that $\hat X(P^*)$ $=$ $X^*(P^*)$ and $\hat Q(\hat X(P^*))$ $=$  $Q^*(P^*)$.  We also consider the concept of equilibrium price when focusing on a two-player game between the firm and the consumer (hence without the social planner). In this context, we view the firm as a {\it leader} which offers a price to the consumer (the {\it follower}), and the endogenous price results from a Stackelberg equilibrium where the leader optimises the expected total cost through the price and its investment rate given the best reaction effort (the distributed generation) of the follower. These two equilibrium issues will be discussed and addressed in Section \ref{secequil}.

\section{Optimal behaviour of the players}
\label{sec:Players}
In this section, we address the optimisation problems for each player (consumer, firm and social planner) and determine their optimal generation investment rate explicitly. 

\subsection{The consumer's problem} \label{secconsu}

\setcounter{equation}{0}
\setcounter{Assumption}{0} \setcounter{Theorem}{0}
\setcounter{Proposition}{0} \setcounter{Corollary}{0}
\setcounter{Lemma}{0} \setcounter{Definition}{0}
\setcounter{Remark}{0}

From the expression \reff{totalcost} of the expected total discounted cost of the consumer, and remo\-ving the term $P_tD$ that does not depend on the installation rate control $\alpha$, the consumer's problem is written as
\beq \label{defVC} 
V^C &=& \inf_{\alpha\in\Ac} \E \Big[ \int_0^\infty e^{-\rho t} f_t(X_t^\alpha,\E[X_t^\alpha],\alpha_t) dt \Big],  
\enq
where $\{f_t(x,\bar x,a), (x,\bar x,a) \in \R^3, t \geq 0\}$ is the $\F^0$-adapted random field defined by
\beq \label{deff}
f_t(x,\bar x,a) &=& ca + \gamma a^2  - (1+\th)P_t x + \eta (x^2 - \bar x^2).
\enq
Here $\Ac$ is the set of admissible controls defined as  the set of $\F$-adapted real-valued processes $\alpha$ s.t.~$\E[\int_0^\infty e^{-\rho t} |\alpha_t|^2 dt]$ $<$ $\infty$. We  assume that the positive discount factor $\rho$ satisfies:
\beq
\label{rhosigma}
\rho & > &  \sigma^2,
\enq
which implies from It\^o's formula on \reff{dynX}, Young's inequality and Gronwall's lemma  that
\beq \label{integX}
\E \Big[ \int_0^\infty e^{-\rho t} |X_t^\alpha|^2 dt \Big] & \leq & \left( x_0^2 + \frac{b}{\epsilon} \int_0^\infty e^{-\rho s} \E[|\alpha_s|^2] ds \right) \int_0^\infty e^{-(\rho - \sigma^2 - \epsilon b)s} ds \; < \; \infty, \,\,\,\,\, 
\enq
for $\epsilon>0$ small enough. We also assume that the price process satisfies the square-integrability condition
\beq \label{integP}
\E \Big[ \int_0^\infty e^{-\rho t} |P_t|^2 dt \Big] & < & \infty, 
\enq
which  ensures by \reff{integX}  that problem \reff{defVC} is well-defined. 
Problem \reff{defVC} is a nonstandard stochastic control problem due to the variance term, which induces {\em a priori} time inconsistency. It belongs to the class of so-called McKean-Vlasov (MKV) control problems where the running cost $f_t$ in \reff{deff} involves the law of the state process $X^\alpha$ in a nonlinear dependence coming from the square of its mean. McKean-Vlasov control problems have received a surge of interest in the last few years in connection with the mean-field game theory initiated by Lasry and Lions, and we refer for example  to \cite{benetal13}, \cite{cardel17}, and \cite{phawei16} for recent works on this topic. Problem \reff{defVC} falls more specifically into  the class of linear-quadratic MKV control problems with random coefficients that have been recently studied in \cite{baspha17}. The method of resolution of \reff{defVC} is based on a mean version of martingale optimality principle in dynamic programming leading to the following verification argument.
 
\begin{Lemma}[{\em Optimality principle}] 
\label{lemverif}
Let $\{ V^\alpha, \alpha \in \Ac\}$  be  a family of $\F$-adapted processes in the form 
$V_t^\alpha$ $=$ $v_t(X_t^\alpha,\E[X_t^\alpha])$, $t\geq 0$, for some random field $\F$-adapted process 
$\{v_t(x,\bar x)$, $(x,\bar x) \in \R^2$, $t \geq 0\}$ satisfying the following conditions:
\begin{itemize}
\item[(i)] there exist a positive constant $C$ and a $\F$-adapted process $(I_t)_{t \geq 0}$ such that $\E[e^{-\rho T} I_T]$ $\rightarrow$ $0$ as $T \to \infty$ and that $|v_t(x,\bar x)| \leq C( I_t + |x|^2 + |\bar x|^2)$, for $t \geq 0$ and $(x,\bar x) \in \R^2$;
\item[(ii)] for all $\alpha$ $\in$ $\Ac$, the map $t$ $\in$ $\R_+$ $\mapsto$ $\E[S_t^\alpha]$, with 
$S_t^\alpha$ $=$ $e^{-\rho t} V_t^\alpha + \int_0^t e^{-\rho s} f_s(X_s^\alpha,\E[X_s^\alpha],\alpha_s) ds$, is nondecreasing;
\item[(iii)] there exists some $\hat\alpha$ $\in$ $\Ac$ such that the map $t$ $\in$ $\R_+$ $\mapsto$ $\E[S_t^{\hat\alpha}]$ is constant.
\end{itemize}
Then $\hat\alpha$ is an optimal control for problem \reff{defVC} and $V^C$ $=$ $\E[v_0(x_0,x_0)]$. Moreover, any other optimal control satisfies the condition (iii). 
 \end{Lemma}
\noindent {\bf Proof.} Let $\alpha \in \aaa$ and $T>0$. Setting $f_t^\alpha$ $:=$ $f_t(X_t^\alpha,\E[X_t^\alpha], \alpha_t)$, from (ii) we get
\beq 
\E[v_0(x_0,x_0)] \; = \; \E[V_0^\alpha] \; = \; \E[S_0^\alpha] \leq \E[S_T^\alpha] \; = \; \E \Big[ e^{-\rho T} V_T^\alpha  +  \int_0^T e^{-\rho s} f_s^\alpha ds \Big]. 
\label{inegv} 
\enq
By (i) and \reff{integX} we have $\E [e^{-\rho T} |V_T^\alpha|] \to 0$ as $T \to \infty$; hence, by sending $T$ to infinity into \reff{inegv} and by using the dominated convergence theorem ($f_t$ satisfies a quadratic growth condition), we have $\E[v_0(x_0,x_0)] \leq \E[\int_0^\infty e^{-\rho s} f_s^\alpha ds]$ and then $\E[v_0(x_0,x_0)]$ $\leq$ $V^C$, since $\alpha$ is arbitrary in $\Ac$. Similarly, by (iii) we obtain $\E[v_0(x_0,x_0)]$  $=$ $\E[\int_0^\infty e^{-\rho s} f_s^{\hat\alpha} dt]$ $\geq$ $V^C$, so that $\hat\alpha$ is an optimal control and $V^C$ $=$ $\E[v_0(x_0,x_0)]$. Suppose now that $\tilde\alpha$ $\in$ $\Ac$ is another optimal control. For all $T>0$ we have
\beqs
\E[S_0^{\tilde\alpha}] \; = \; V^C &=& \E\big[\int_0^\infty e^{-\rho t} f_t^{\tilde\alpha} dt \big] \; = \; \E[ S_T^{\tilde\alpha} ] 
+ \E\big[ \int_T^\infty f_s^{\tilde\alpha} ds - e^{-\rho T} V_T^{\tilde\alpha} \big],
\enqs
so that, by sending $T$ to infinity, we get $\E[S_0^{\tilde\alpha}]$ $=$ $\lim_{T\rightarrow\infty}\E[S_T^{\tilde\alpha}]$; since the map $t$ $\mapsto$ $\E[S_t^{\tilde\alpha}]$ is nondecreasing, this shows that this map is actually constant. 
\ep

\begin{Remark} \label{remmar}
{\rm The verification conditions (ii) and (iii) in Lemma \ref{lemverif} are weaker than the usual conditions for martingale optimality principle in stochastic control (see \cite{elk81}), which require  that $S^\alpha$ is a supermartingale for all $\alpha$, and a martingale for some $\hat\alpha$. 
In the case of a standard (without McKean-Vlasov dependence) stochastic control problem, one looks for a family of processes $V^\alpha$ in the form $V_t^\alpha$ $=$ $v_t(X_t^\alpha)$, and the classical martingale optimality principle is used in practice by applying It\^o's formula to $S_t^\alpha$ and then 
deriving a dynamic programming Hamilton-Jacobi-Bellman equation   (or backward stochastic differential equation in the case of random coefficients) for the (random) value function $v_t(x)$  
from the non-positivity of the drift of $S^\alpha$ for any $\alpha$ (supermartingale condition), and the cancellation of the drift of 
$S^{\hat\alpha}$ (martingale condition). In our McKean-Vlasov framework, this martingale optimality principle cannot be used in practice for finding the solution to the control problem, and we instead exploit the weaker version in Lemma \ref{lemverif} where the optimality principle is formulated on the mean of 
$S^\alpha$. This is the purpose of the next paragraph.
}
\ep
\end{Remark}

\subsubsection{Optimal distributed generation}

We explain in this paragraph how to exploit the mean martingale optimality principle in Lemma \ref{lemverif} for solving problem \ref{defVC}. 
The technical details and complete rigorous derivation are postponed in Appendix \ref{AppA1}. 
Given the linear-quadratic structure of our MKV control problem, we look for a candidate random value function $v_t(x,\bar x)$ in the form: 
\beq
\label{CONScandidate}
v_t(x,\bar x) &=& K_t ( x- \bar x)^2 + \Lambda_t \bar x^2  + Y_t x + R_t, 
\enq
for some $\F$-adapted processes $K$, $\Lambda$, $Y$ and $R$ to be determined. Because the randomness in the coefficients of the consumer's problem comes only from the $\F^0$-adapted price process $P$ in the linear term in $f_t$, we search for deterministic $K$, $\Lambda$, and $\F^0$-adapted processes $Y$ and $R$, hence in the evolution form (recall that $\F^0$ is the filtration of the Brownian motion $W^0$): 
\beqs
dK_t \; = \; \dot K_t dt, & &  d \Lambda_t \; = \; \dot \Lambda_t dt, \\
dY_t \; = \; \dot Y_t dt + Z_t^Y dW_t^0, & & dR_t \; = \; \dot R_t dt + Z_t^R dW_t^0,
\enqs
for some deterministic processes $\dot K$, $\dot \Lambda$, and $\F^0$-adapted processes $\dot Y$, $\dot \Gamma$, $\dot R$, $Z^Y$, $Z^R$. Next, applying It\^o's formula to $S_t^\alpha$ $=$ $e^{-\rho t} v_t(X_t^\alpha,\E[X_t^\alpha])$ $+$ 
$\int_0^t e^{-\rho s} f_s(X_s^\alpha,\E[X_s^\alpha],\alpha_s)ds$, for $\alpha$ $\in$ $\Ac$, and taking the expectation, we get
\beqs
d\E[S_t^\alpha] &=& e^{-\rho t} \E[ \Dc_t^\alpha ] dt, 
\enqs
for some $\F$-adapted processes $\Dc^\alpha$ with 
\begin{equation}
\label{CONSdefD}
\E[ \Dc_t^\alpha] = \E \Big[ - \rho v_t(X_t^\alpha,\E[X_t^\alpha]) + \frac{d}{dt} \E\big[ v_t(X_t^\alpha,\E[X_t^\alpha])\big] + 
f_t(X_t^\alpha,\E[X_t^\alpha],\alpha_t) \Big], \;\; t \geq 0. 
\end{equation}
According to Lemma \ref{lemverif}, we have to determine $K$, $\Lambda$, $Y$, $\Gamma$, and $R$ so that
\beq \label{conDc}
 \; \E[\Dc_t^\alpha] \; \geq \; 0,  \; t \geq 0, \forall \alpha\in \Ac, &\mbox{and}&  \; \E[\Dc_t^{\hat\alpha}] \; = \; 0, \; t \geq 0, 
\mbox{ for some } \;  \hat\alpha \in \Ac,
\enq
which ensures that the mean optimality principle conditions (ii) and (iii) are satisfied, and then $\hat\alpha$ will be an optimal control.  For this, we apply It\^o's formula to $v_t(X_t^\alpha,\E[X_t^\alpha])$ recalling the quadratic form of $v_t(\cdot,\cdot)$, the dynamics of $X^\alpha$ in \reff{dynX}, hence of its mean: 
$d\E[X_t^\alpha]$ $=$ $b \E[\alpha_t] dt$, and after some calculations together with square completion, we obtain: 
\beq
\label{CONSito}
\E[\Dc_t^\alpha]  &=& \E \Big[ \gamma \big( \alpha_t  - A(X_t^\alpha,\E[X_t^\alpha],K_t,\Lambda_t,Y_t) \big)^2  \\
& & \;\;\;\;\; + \;  F(K_t, \dot K_t) \big(X_t^\alpha  - \E[X_t^\alpha] \big)^2  + G(K_t, \Lambda_t, \dot \Lambda_t) \big(\E[X_t^\alpha] \big)^2 \nonumber\\
& & \;\;\;\;\; + \; H_t(K_t, \Lambda_t, Y_t,\eee[Y_t],\dot Y_t) X_t^\alpha + \; M(Y_t, R_t, \dot R_t)  \Big], \;\;\; t \geq 0, \;  \forall  \alpha\in\Ac, \nonumber
\enq
where $A$, $F$, $G$, $H$, and $M$ are explicit functions of their arguments (see Appendix \reff{AppA1}). It follows that condition \reff{conDc} is realised whenever
the following equations are satisfied
\beq
\label{CONSzerocoeff}
F(K_t, \dot K_t) \; = \; 0,  & & G(K_t, \Lambda_t, \dot \Lambda_t)  \; = \; 0, \;\;\; t \geq 0,  \\
H_t(K_t, \Lambda_t, Y_t,\eee[Y_t],\dot Y_t)  \; = \; 0, & &   M(Y_t, R_t, \dot R_t) \; = \; 0, \;\;\; t \geq 0, \nonumber
\enq
and 
\beq
\label{CONSimplicitcontr}
\hat\alpha_t &=&  A(X_t^{\hat \alpha},\E[X_t^{\hat \alpha}],K_t,\Lambda_t,Y_t), \;\;\; t \geq 0.   
\enq
The explicit resolution of these equations leads to our first main result which is a closed-form expression of the optimal distributed generation production for the consumer. 

\begin{Theorem} \label{theoconsu}
Let \eqref{rhosigma} and \eqref{integP} hold. The optimal solar panel installation rate is given by 
\beq 
\hat\alpha_t &=&  -\frac{bK}{\gamma} ( \hat X_t - \E[\hat X_t]) -\frac{b\Lambda}{\gamma} \E[\hat X_t] - \frac{c}{2\gamma} \frac{\rho}{\rho + \frac{b^2\Lambda}{\gamma}} \nonumber \\
& & \; + \;  \frac{b(1+\th)}{2\gamma}   \int_t^\infty e^{-(\rho+\frac{b^2K}{\gamma})(s-t)} \E[P_s  | \Fc_t^0] ds  \nonumber  \\
& & \; + \; \frac{b(1+\th)}{2\gamma} \int_t^\infty \Big( e^{-(\rho + \frac{b^2\Lambda}{\gamma})(s-t)} - e^{-(\rho + \frac{b^2K}{\gamma}) (s-t)} \Big) 
\E[P_s] ds, \;\;\; t \geq 0, 
\label{CONScontropt}
\enq
with the optimal cumulative production of distributed generation $\hat X$ given in mean by
\beq
\label{CONSmeanprod}
\E[\hat X_t ] &=&  x_0 e^{ - \frac{b^2\Lambda}{\gamma} t}  + \frac{b^2(1+\th)}{2\gamma} \int_0^t  e^{-\frac{b^2\Lambda}{\gamma}(t-s)}  
 \int_s^\infty  e^{-(\rho+\frac{b^2\Lambda}{\gamma})(u-s)} \E[ P_u]  du \, ds \nonumber \\
& & \;\;\; - \; \frac{\rho c}{b} \; \frac{1}{2 \Lambda(\rho + \frac{b^2\Lambda}{\gamma})}  \big( 1 - e^{-\frac{b^2\Lambda}{\gamma}t}   \big), \;\;\; t \geq 0,
\enq
where the positive constants $K$ $>$ $\Lambda$ are equal to
\beq
\label{CONSdefKL}
K = \gamma \frac{-(\rho - \sigma^2) + \sqrt{(\rho - \sigma^2)^2 + \frac{4b^2\eta}{\gamma}}}{2b^2},
\qquad\qquad
\Lambda = \gamma \frac{-\rho + \sqrt{\rho^2 + \frac{4b^2\sigma^2K}{\gamma}}}{2b^2}.  
\enq
\end{Theorem}

\vspace{3mm}
\noindent {\bf Interpretation and comments.} 
The expression \reff{CONScontropt} of the optimal solar panels installation rate consists in several explicit terms that have the following interpretation: the first term shows the  mean-reversion of the production towards its mean due to the variance penalization in the criterion, and we notice that the speed of mean-reversion, proportional to $K$, increases with the penalty parameter $\eta$. The second term is related to the generation uncertainty of the solar panels, and decreases the installation rate all the more so as the current level of generation (on average) and its volatility $\sigma$ are high. The third term, which is negative, expresses the obvious feature that the higher is the fixed cost $c$ on solar panels installation, the smaller is the optimal rate of production. Finally, the fourth and fifth terms, depending on the electricity price, quantify the natural intuition that the higher is the expectation on the future price, the higher is the incentive to invest in photovoltaics for self-generation. 

In particular,  if we assume that the price $P_s$ is constant equal to $P$, the expression  \reff{CONScontropt}  takes the following form:
\beqs 
\hat\alpha_t &=&   -\frac{bK}{\gamma} ( \hat X_t - \E[\hat X_t]) -\frac{b\Lambda}{\gamma} \E[\hat X_t] 
+ \frac{b}{2\gamma} \frac{1}{\rho + \frac{b^2\Lambda}{\gamma}} \Big( (1+\theta) P - \frac{\rho c}{b}\Big).
\enqs
The last term for the  optimal investment rate  is a constant rate of installation whose sign depends on the relative value of buying centralised energy at price $(1+\theta) P$ compared to the annuity of investment in distributed energy $ \frac{\rho c}{b}$. If centralised energy is more costly than distributed energy, then the consumer installs distributed energy at a constant rate.

\subsubsection{Long-term self-generation}
\label{ssec:ltsg}
We address the long-term behaviour of the optimal distributed generation when the price process has a stationary level, that is,~$\E[P_t]$ $\rightarrow$ $\bar P$ for some positive constant $\bar P$ $>$ $0$, as $t$ goes to infinity. This stationarity occurs for example when the price process is a martingale, in which case $\bar P$ $=$ $P_0$, or when it is mean-reverting (following e.g.~an Ornstein-Uhlenbeck or a Cox-Ingersoll-Ross process) with a stationary level of $\bar P$. 
\begin{Proposition} \label{propconsulimit}
\label{prop:limitscons}
Let \eqref{rhosigma} and \eqref{integP} hold and assume that there exists $\bar P$ $>$ $0$ s.t.~$\lim_{t\rightarrow\infty}\E[P_t]$ $=$ $\bar P$. Then, the optimal cumulative production of distributed generation admits a stationary level:
\beq
\label{CONSdefXinfty}
\lim_{t\rightarrow\infty} \E[\hat X_t] \; = \;  \frac{ (1+\th) \bar P - \frac{\rho c}{b}}{2 \sigma^2 K} \; =: \; \overline{\hat X_\infty}(\bar P),  & \mbox{ and so } & 
\lim_{t\rightarrow\infty} \E[\hat \alpha_t] \; = \;  0.
\enq
\end{Proposition}
\noindent {\bf Proof.} The first limit follows immediately  from \eqref{CONSmeanprod} by noting that 
\beqs
& & \int_0^t  e^{-\frac{b^2\Lambda}{\gamma}(t-s)}  \int_s^\infty  e^{-(\rho+\frac{b^2\Lambda}{\gamma})(u-s)} \E[ P_u]  du \, ds \\
&=&  \int_0^\infty \int_0^\infty 1_{ \tau \in [0,t]} e^{-\frac{b^2\Lambda}{\gamma}\tau}  
e^{-(\rho+\frac{b^2\Lambda}{\gamma})r} \E[ P_{t-\tau+r}]  d\tau \, dr \\
& \stackrel{t\rightarrow\infty}{\longrightarrow} &   \int_0^\infty \int_0^\infty e^{-\frac{b^2\Lambda}{\gamma}\tau}    e^{-(\rho+\frac{b^2\Lambda}{\gamma})r}  \bar P d\tau dr \; = \; 
\bar P \frac{\gamma}{b^2 \sigma^2 K},
\enqs
by the dominated convergence theorem, and recalling that $\Lambda(\rho \gamma + b \Lambda^2) = \gamma \sigma^2K$. As for the second limit, notice that $\eee[\hat X_t] - x_0 = b \int_0^t \eee[\hat \h_s] ds$, so that $\int_0^\infty \eee[\hat \h_s] ds$ is finite, which implies $\lim_{t\rightarrow\infty} \E[\hat \alpha_t] = 0$.
\ep

\vspace{3mm}
\noindent {\bf Interpretation and comments.}  The spread $(1+\theta) \bar P - \frac{\rho c}{b}$ is the difference for the consumer between the cost of centralised energy and distributed energy ($\rho c/b$ is the annuity in \euro/MW/year of distributed energy). The consumer will invest in distributed energy only if this difference is positive. The consumer's investment will be proportional to this spread and inversely proportional to the volatility of the distributed generation (for reasonable parameter values, it is seen that $K$ weakly depends on  the volatility).  The case of a negative spread leads to a spurious negative long-term installed capacity because we allow the consumer to sell her generation to the firm. However, this long-term behaviour does not depend on the consumer's demand. Since the consumer is not compelled to make her generation close to her demand, the investment choice is purely driven by the relative prices of centralised versus distributed generation. Thus, high centralised energy prices might induce an over-investment in distributed generation.

\vspace{3mm}\noindent {\bf Numerical illustration.} To fix an idea of the order of magnitude of possible investment in distributed generation, consider the following values based on common order of magnitudes for the parameter of the model: $\rho =0.1$ per annum, $\sigma = 0.3$ per annum, $b=0.15$, a unitary cost $c=3000$~\euro/kW, a price of centralised energy of $\bar P$ $=$ 80~\euro/MWh and a transmission cost $\theta = 2$. The market impact parameter is taken to be small at a level of $\gamma=1$~\euro/MW$^2$/year to avoid overestimating its effect as well as the cost of variance of distributed generation that we take to be $\eta = 8760 \times 10$~\euro/MW$^2$/year. The resulting spread price $(1+\th) \bar P - \frac{\rho c}{b}$ expressed in \euro/MWh equals 11~\euro/MWh and the consumer will invest $\overline{\hat X_\infty}(\bar P)$ $=$ $288$~MW in solar energy. Dividing the market impact parameter $\gamma$ by 100 results in an installed distributed generation multiplied by 10. Multiplying the penalty parameter of the variance $\eta$ by 100 results in an installed distributed generation divided by 10.

\subsubsection{Mean-reverting price process}
We illustrate numerically the above results on the optimal distributed generation for the consumer when the price process is mean-reverting. Hence we consider dynamics in the form 
\beqs
dP_t &=& \kappa(\bar P - P_t) dt + \xi_t dW_t^0, 
\enqs
for some constants $\kappa$ $\geq$ $0$, $\bar P$ $>$ $0$, and where the (possibly random) volatility $\xi_t$ of the price process is a $\F^0$-adapted positive process such that
\begin{equation*}
\eee \Big[ \int_0^\infty e^{-\rho s} |\xi_s|^2 ds \Big] < \infty.
\end{equation*}
This modelling includes the cases where $P$ is a martingale ($\kappa$ $=$ $0$), $P$ is an Ornstein-Uhlenbeck process ($\xi_t$ $=$ $\xi$ positive constant), $P$ is governed by a Cox-Ingersoll-Ross process ($\xi_t$ $=$ $\xi\sqrt{P_t}$ for some $\xi$ $>$ $0$ that satisfies $2\kappa\bar P$ $>$ $\xi^2$). For such a mean-reverting process, we have
\beqs
\E[P_s | \Fc_t^0] &=& \bar P + (P_t - \bar P) e^{-\kappa(s-t)}, \;\;\; 0 \leq t \leq s;
\enqs
the optimal solar panel installation rate in Theorem \ref{theoconsu} can be written explicitly in terms of the current self-production level $\hat X_t$, the initial level $x_0$, the current price $P_t$, the initial value $P_0$, and the stationary level $\bar P$. If $b^2\Lambda-\kappa\gamma \neq 0$, then the optimal control is
\beqs
\hat\alpha_t &=&  - \frac{bK}{\gamma}\hat X_t   + \frac{b(K-\Lambda)}{\gamma} x_0  e^{-\frac{b^2\Lambda}{\gamma} t} \\
& & \; + \;   \frac{b(1+\th)}{2(b^2 K + (\kappa + \rho)\gamma)} (P_t - \bar P) \\
& & \; + \;  \frac{b^3(1+\th)(K-\Lambda)}{2(b^2 \Lambda + (\kappa + \rho)\gamma)} 
\Big[  \frac{1}{b^2\Lambda-\kappa\gamma}\big(e^{-\kappa t} -   e^{-\frac{b^2\Lambda}{\gamma} t} \big)    +       
\frac{1}{b^2K + (\kappa+\rho)\gamma} e^{-\kappa t}   \Big](P_0-\bar P)  \\
& & \; + \;  \frac{b (1+\th) \bar P - \rho c}{2\Lambda(b^2\Lambda + \rho\gamma)} \Big[ \Lambda + (K - \Lambda \big) (1- e^{-\frac{b^2\Lambda}{\gamma} t}) \Big], \;\;\; t \geq 0,
\enqs
and the average self-production level is:
\beqs
\E[\hat X_t] &=&  \frac{(1+\th) \bar P - \frac{\rho c}{b}}{2\sigma^2 K} + 
\frac{(1+\th)(P_0-\bar P)b^2\gamma}{2(b^2\Lambda + (\kappa+\rho)\gamma)(b^2\Lambda - \kappa \gamma)} e^{-\kappa t} \\
& & \; + \; \Big( x_0 - \frac{(1+\th) \bar P - \frac{\rho c}{b}}{2\sigma^2 K} - \frac{(1+\th)(P_0-\bar P)b^2\gamma}{2(b^2\Lambda + (\kappa+\rho)\gamma)(b^2\Lambda - \kappa \gamma)} \Big)
e^{-\frac{b^2\Lambda}{\gamma} t}, \;\;\; t \geq 0.
\enqs
If $b^2\Lambda-\kappa\gamma = 0$, then the optimal control is
\begin{align*}
\hat\alpha_t =& - \frac{bK}{\gamma}\hat X_t   + \frac{b(K-\Lambda)}{\gamma} x_0  e^{-\frac{b^2\Lambda}{\gamma} t} \\
& + \frac{b(1+\theta)}{2(b^2 K + (\kappa + \rho)\gamma)} (P_t - \bar P) \\
& +\frac{b^3(1+\theta)(K-\Lambda)}{2(b^2 \Lambda + (\kappa + \rho)\gamma)} 
\Big[ \frac{t}{\gamma} +       
\frac{1}{b^2K + (\kappa+\rho)\gamma} \Big] e^{-\frac{b^2\Lambda}{\gamma} t} (P_0-\bar P)  \\
& + \frac{b(1+\theta) \bar P - \rho c}{2\Lambda(b^2\Lambda + \rho\gamma)} \Big[ \Lambda + (K - \Lambda \big) (1- e^{-\frac{b^2\Lambda}{\gamma} t}) \Big], \;\;\; t \geq 0,
\end{align*}
and the average self-production level is:
\beqs
\E[\hat X_t] =  \frac{(1+\th) \bar P - \frac{\rho c}{b}}{2\sigma^2 K} + \Big( x_0 - \frac{(1+\th) \bar P - \frac{\rho c}{b}}{2\sigma^2 K} + \frac{(1+\th)(P_0-\bar P)b^2}{2(b^2\Lambda + (\kappa+\rho)\gamma)} t \Big)
e^{-\frac{b^2\Lambda}{\gamma} t}, \;\;\; t \geq 0.
\enqs
Further, $\E[\hat X_t]$ converges exponentially fast to its stationary level $\bar x(\bar P)$. 
We  plot in Figure~\ref{fig:onetrajectory} a sample trajectory  of the optimal control $\hat\alpha$, of the optimal production levels $\hat X$, and of the price process (an Ornstein-Uhlenbeck model for $P$). We also illustrate the convergence of the optimal production to its stationary level in Figure~\ref{fig:conv}. Further, we illustrate the effect of the variance penalty parameter $\eta$ by plotting 50 trajectories of the optimal cumulative self-production $\hat X$ in Figures~\ref{fig:eta2} and \ref{fig:eta8}. 

\vfill

\begin{figure}[h!]
	\begin{minipage}{0.45\textwidth}
		\centering
		\includegraphics[width=\textwidth]{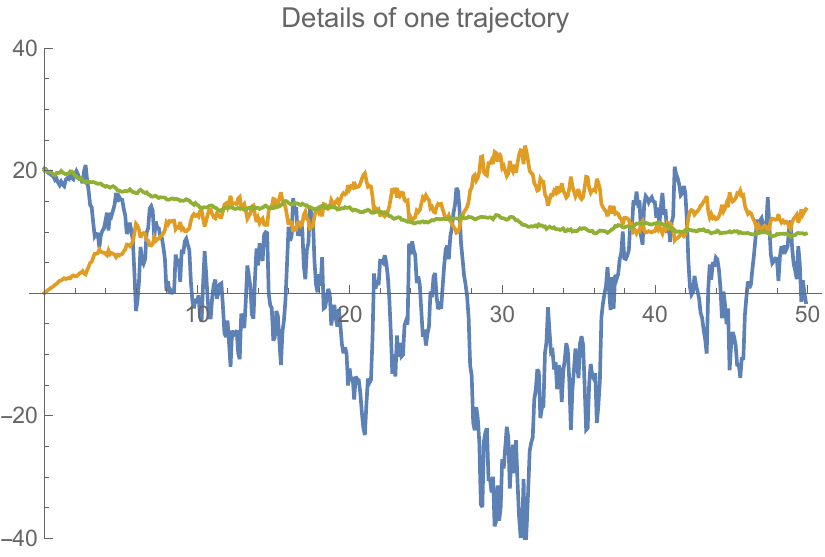}
		\caption{\small{evolution of $\hat\alpha$ (blue), $\hat X$ (orange) and $P$ (green)}.}
		\label{fig:onetrajectory}
	\end{minipage}
	\hfill
	\begin{minipage}{0.45\textwidth}
		\centering
		\includegraphics[width=\textwidth]{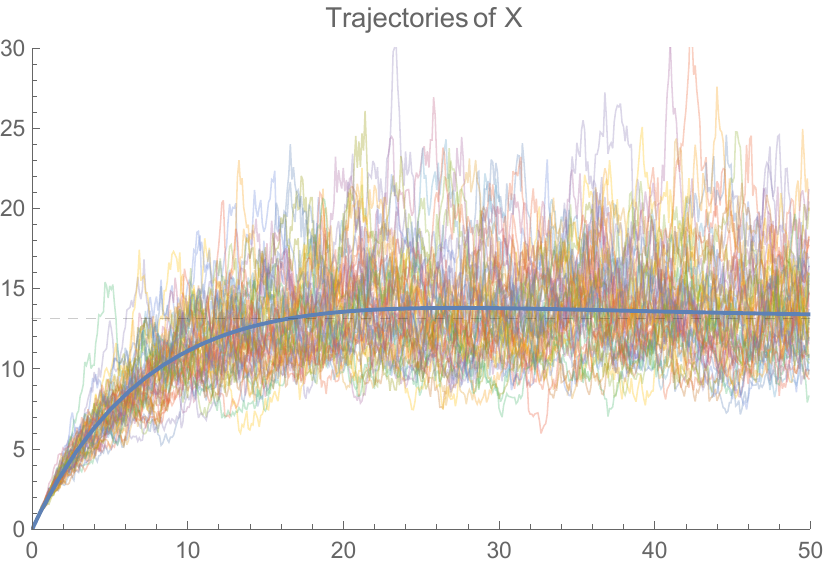}
		\caption{\small{convergence of the trajectories of $\hat X$ for $\eta$ $=$ $4$.}}
		\label{fig:conv}
	\end{minipage}
\end{figure}

\newpage

\begin{figure}[h!]
	\begin{minipage}{0.45\textwidth}
		\centering
		\includegraphics[width=\textwidth]{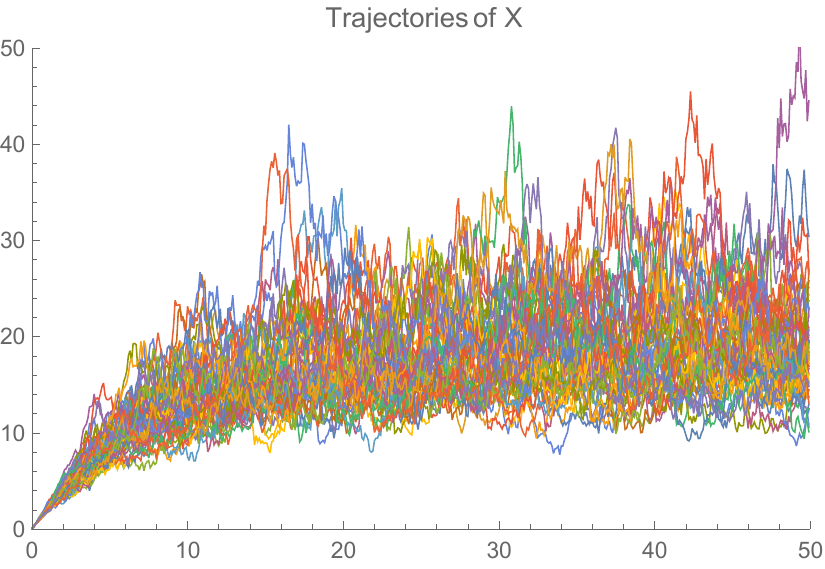}
		\caption{\small{trajectories of $\hat X$ for $\eta$ $=$ $2$.}}
		\label{fig:eta2}
	\end{minipage}
	\hfill
	\begin{minipage}{0.45\textwidth}
		\centering
		\includegraphics[width=\textwidth]{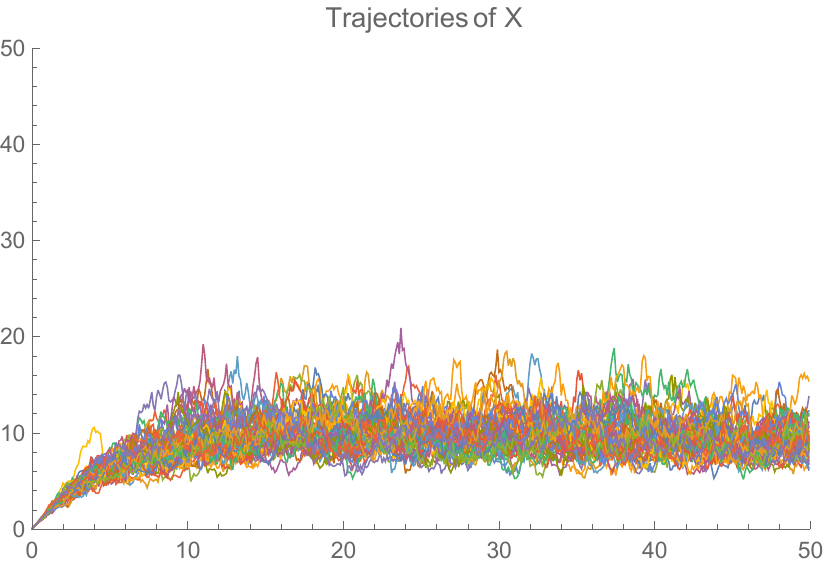}
		\caption{\small{trajectories of $\hat X$ for $\eta$ $=$ $8$.}}
		\label{fig:eta8}
	\end{minipage}
\end{figure}

\subsection{The firm's problem} \label{seccomp}

We start from the expression \reff{costenergy} of the expected total discounted cost of the firm and remove the terms that are not affected by the centralised production. Thus, given a distributed generation production $X^\alpha$, we wrtite the firm's problem as 
\beq \label{defVF}
V^F(X^\alpha) &=& \inf_{\nu \in \Vc} \E \Big[ \int_0^\infty e^{-\rho  t} \ell_t(Q^\nu_t, \nu_t) dt \Big],
\enq
where $\{\ell_t(q,u), (q,u) \in \R^2, t \geq 0\}$ is the $\F^0$-adapted random field defined by
\beq \label{deff2}
\ell_t(q,u) &=& h u^2 - 2\lambda \big( 1  - \frac{\pi_1}{2\lambda D} \big)(D - X_t^\alpha) q + \lambda q^2.
\enq
Here $\Vc$ is the set of $\F$-adapted real-valued processes $\nu$ s.t.~$\E[\int_0^\infty e^{-\rho t} |\nu_t|^2 dt]$ $<$ $\infty$, which implies that $\E[\int_0^\infty e^{-\rho t} |Q_t^\nu|^2 dt]$ $<$ $\infty$. Because $\E[\int_0^\infty e^{-\rho t}|X_t^\alpha|^2 dt]$ $<$ $\infty$ for $\alpha$ $\in$ $\Ac$, the problem \reff{defVF} is well-defined. 
Problem \reff{defVF} fits into the class of linear-quadratic stochastic control problems with random coefficients and on infinite horizon.  
We can solve the problem by following the same martingale optimality principle as  for the consumer's problem, although without the McKean-Vlasov dependence (see Remark \ref{remmar}). We provide a closed-form expression for the optimal centralised generation production of the firm, as well as its long-term behaviour. The proof is reported in Appendix \ref{AppA2}. 
\begin{Theorem} \label{theo2}
Given a distributed generation $X^\alpha$, the optimal rate $\hat\nu(X^\alpha)$ for centralised generation is given by
\begin{equation}
\label{FIRMoptcontr}
\hat\nu_t(X^\alpha) = - \frac{\tilde K}{h} \hat Q_t(X^\alpha) 
+ \frac{\lambda}{h} \big( 1  - \frac{\pi_1}{2\lambda D} \big) \int_t^\infty e^{-(\rho+\frac{\tilde K}{h})(s-t)} \E[D - X_s^\alpha|\Fc_t] ds, \quad t \geq 0, 
\end{equation}
with the optimal cumulative production $\hat Q(X^\alpha)$ given in mean by
\begin{equation}
\label{FIRMexpect}
\E[\hat Q_t(X^\alpha)] = q_0 e^{-\frac{\tilde K}{h} t} + \frac{\lambda}{h} \big( 1 - \frac{\pi_1}{2\lambda D} \big) \int_0^t e^{-\frac{\tilde K}{h}(t-s)} 
\int_s^\infty  e^{-(\rho+\frac{\tilde K}{h})(u-s)} \E[D - X_u^\alpha] du \, ds, \quad t \geq 0,  
\end{equation}
and where $\tilde K$ is the positive constant equal to
\beq
\label{FIRMtildeK}
\tilde K &=&  \frac{-\rho  + \sqrt{\rho^2 + \frac{4\lambda}{h}}}{\frac{2}{h}}. 
\enq
Moreover, when the distributed production process has a stationary level, that is~$\lim_{t\rightarrow\infty}\E[X^\h_t]$ $=$ $\bar X_\infty^\alpha$ for some 
$\bar X_\infty^\alpha$ $>$ $0$, then the optimal cumulative production of centralised generation also admits a stationary level:
\beq
\label{FIRMlimits}
\lim_{t\rightarrow\infty} \E[\hat Q_t(X^\alpha)] = \big( 1  - \frac{\pi_1}{2\lambda D} \big)(D- \bar X_\infty^\alpha) \quad\mbox{ and so }\quad 
\lim_{t\rightarrow\infty} \E[\hat \nu_t(X^\alpha)] \; = \;  0.
\enq
In particular, if $\lim_{t\rightarrow\infty}\E[P_t]$ $=$ $\bar P$ exists and we consider the optimal distributed generation $\hat X$ $=$ $X^{\hat\alpha}$, then the limit of the centralised  production $\hat Q$ $=$   $\hat Q(\hat X)$  reads as
\beq
\label{FIRMdefQinfty}
\lim_{t\rightarrow\infty} \E[\hat Q_t] = \big( 1  - \frac{\pi_1}{2\lambda D} \big)\big(D- \overline{\hat X_\infty }(\bar P)\big) \; =: \;   \overline{\hat Q_\infty}(\bar P),
\enq
where $\overline{\hat X_\infty }$ is defined in \eqref{CONSdefXinfty}.
\end{Theorem} 
  
\vspace{3mm}
\noindent {\bf Interpretation and comments.} In the long-term the firm invests the fraction $1 - \pi_1/(2\lambda D)$ of the residual demand of the consumer. Nevertheless, the ratio $\pi_1/(2\lambda D)$ is close to zero. Indeed, as pointed in Section \ref{sec:Model}, if we consider that centralised energy is produced by coal-fired plants (1 MWh will generate 1 tonne of carbon emission), then $\pi_1 \approx 100$~\euro/MWh. Moreover the demand is of the order of magnitude of 50000 MW for a country like France or a state like California. The parameter $\lambda$ is also significantly large to compel the firm to satisfy consumer demand. Thus, it results that $\pi_1/(2\lambda D)$ is very small compared to 1 and that the firm invests a capacity equal to consumer's residual demand.

\subsection{The social planner's problem} 
\label{secsocial}

The social planner's problem, which consists in minimising the expected total costs of the consumer and the firm, is written as
\beq \label{defVS}
V^S &=& \inf_{\alpha\in\Ac,\, \nu\in\Vc} \E \Big[ \int_0^\infty e^{-\rho t}  g_t(X_t^\alpha,\E[X_t^\alpha],Q_t^\nu,\alpha_t,\nu_t)  dt \Big],
\enq
where  $\{g_t(x,\bar x,q,a,u), (x,\bar x,q,a,u) \in \R^5, t \geq 0\}$ is the $\F^0$-adapted random field defined by
\beqs
g_t(x,\bar x,q,a,u) &=& ca+\gamma a^2  + h u^2  +  (\eta + \lambda - \frac{\pi_0}{D}) x^2  - \eta \bar x^2  + \lambda q^2 
+ (2\lambda - \frac{\pi_1}{D}) xq  \\
& & \qquad+ \;  (\pi_0 -  2 \lambda D - \th P_t) x + (\pi_1 - 2\lambda D) q.  
\enqs
Problem \reff{defVS} belongs to the class of  linear quadratic McKean-Vlasov control problems in dimension 2, and is also solved by the same martingale optimality principle as described in Section \ref{secconsu}. The next result, whose proof is reported again in Appendix \ref{AppA3}, gives the closed-form expression for the optimal distributed and centralised generation determined by the social planner. 
\begin{Theorem} \label{theosocial}
Let \eqref{rhosigma} and \eqref{integP} hold and assume that
\begin{equation}
\label{CondParamSocPl}
D > \max \Big\{\frac{\pz}{\lambda}; \frac{\po^2}{4\lambda(\po-\pz)}\Big\}.
\end{equation}
Then, the optimal distributed and centralised generation for the social planner are given by
\begin{align}
\label{SOCPLoptimal}
\alpha_t^* =& -\frac{b K^{11}}{\gamma} (X^*_t-\eee[X^*_t]) - \frac{b K^{12}}{\gamma} (Q^*_t-\eee[Q^*_t]) -\frac{b \Lambda^{11}}{\gamma}\eee[X^*_t] -\frac{b \Lambda^{12}}{\gamma}\eee[Q^*_t] - \frac{b}{2\gamma} \tilde Y^1_t, \nonumber
\\
\nu_t^* =&  -\frac{K^{12}}{h} (X^*_t-\eee[X^*_t]) - \frac{K^{22}}{h} (Q^*_t-\eee[Q^*_t]) -\frac{\Lambda^{12}}{h}\eee[X^*_t] -\frac{\Lambda^{22}}{h}\eee[Q^*_t] - \frac{1}{2h} \tilde Y^2_t, \;\;\; t \geq 0,
\end{align}
where $ \{ K^{ij} \}_{1\leq i,j \leq 2}, \{ \Lambda^{ij} \}_{1\leq i,j \leq 2}$ are the unique symmetric positive-definite solutions to the {stochastic algebraic Riccati equation} (SARE) systems:
\begin{equation}
\label{SOCPLdefKL}
\begin{aligned}
&\begin{cases}
\frac{b^2}{\gamma}(K^{11})^2 + \frac{1}{h}(K^{12})^2 +\rho K^{11} - \sigma^2 K^{11} - \lambda - \eta + \frac{\pi_0}{D}  = 0,
\\
\frac{b^2}{\gamma}K^{11}K^{12} + \frac{1}{h}K^{12}K^{22} +\rho K^{12} - \lambda + \frac{\pi_1}{2D} = 0,
\\
\frac{b^2}{\gamma}(K^{12})^2 + \frac{1}{h}(K^{22})^2 +\rho K^{22} -\lambda = 0,
\end{cases}
\\
&\begin{cases}
\frac{b^2}{\gamma}(\Lambda^{11})^2 + \frac{1}{h}(\Lambda^{12})^2 +\rho \Lambda^{11} - \sigma^2 K^{11} - \lambda + \frac{\pi_0}{D} = 0,
\\
\frac{b^2}{\gamma}\Lambda^{11}\Lambda^{12} + \frac{1}{h}\Lambda^{12}\Lambda^{22} +\rho \Lambda^{12}  - \lambda + \frac{\pi_1}{2D} = 0,
\\
\frac{b^2}{\gamma}(\Lambda^{12})^2 + \frac{1}{h}(\Lambda^{22})^2 +\rho \Lambda^{22} -\lambda = 0,
\end{cases}
\end{aligned}
\end{equation}
and where $(\tilde Y^1, \tilde Y^2)=:\tilde Y$ is a stochastic process defined by
\begin{equation*}
\tilde Y_t = \int_t^\infty e^{-\Xi_1(s-t)}  \eee[ T_s| \mathcal{F}^0_t] ds + \int_t^\infty \Big( e^{-\Xi_2(s-t)} - e^{-\Xi_1(s-t)} \Big) \eee[T_s] ds + \Xi_2^{-1}
\begin{pmatrix}
\frac{\rho c}{b} \\
0 \\
\end{pmatrix},
\end{equation*}
with $\Xi_1,\Xi_2$ constants defined by
\begin{equation*}
\Xi_1 =
\begin{pmatrix}
\frac{b^2K^{11}}{\gamma}+\rho & \frac{b^2K^{12}}{h} \\
\frac{b^2K^{12}}{\gamma} & \frac{b^2K^{22}}{h}+\rho \\
\end{pmatrix},
\qquad
\Xi_2 =
\begin{pmatrix}
\frac{b^2\Lambda^{11}}{\gamma}+\rho & \frac{b^2\Lambda^{12}}{h} \\
\frac{b^2\Lambda^{12}}{\gamma} & \frac{b^2\Lambda^{22}}{h}+\rho \\
\end{pmatrix}
\end{equation*}
and $T$ is a stochastic process in $\R^2$ defined by
\begin{equation*}
T_t=
\begin{pmatrix}
\pi_0 -2\lambda D - \th P_t \\
\pi_1 -2\lambda D \\
\end{pmatrix},
\;\;\; t \geq 0.
\end{equation*}
The associated cumulative productions $X^*$ and $Q^*$ are given in mean by
\begin{align}
\label{SOCPLproduction}
\eee[X^*_t] =& \Gamma^{11}_t x_0 + \Gamma^{12}_t q_0 + \int_0^t \Big( \frac{b^2}{2\gamma} \Gamma^{11}_{t-s} \eee[\tilde Y^1_s] +  \frac{1}{2h}\Gamma^{12}_{t-s} \eee[\tilde Y^2_s] \Big) ds, \\
\eee[Q^*_t] =&  \Gamma^{21}_t x_0 + \Gamma^{22}_t q_0 + \int_0^t \Big( \frac{b^2}{2\gamma} \Gamma^{21}_{t-s} \eee[\tilde Y^1_s] +  \frac{1}{2h}\Gamma^{22}_{t-s} \eee[\tilde Y^2_s] \Big) ds,
\nonumber
\end{align}
where $\Gamma_\tau= e^{-(\Xi'_2 - \rho \, \text{id}) \tau}$ for $\tau >0$ and $\Xi'_2$ is the transpose of $\Xi_2$. Moreover, when the price process has a stationary level, i.e.~$\lim_{t\rightarrow\infty}\E[P_t]$ $=$ $\bar P$ for some $\bar P$ $>$ $0$, then the 
optimal cumulative productions  also admit stationary levels:
\beq 
\lim_{t\rightarrow\infty} \E[X_t^*] &=& \frac{2\lambda D^2  \big( 2 \pi_1 - \pi_0 + \th \bar P- \frac{\rho c}{b} \big) - D\pi_1^2 }{4\lambda D\big( \pi_1 - \pi_0 +\sigma^2K^{11}D \big) - \pi_1^2} \; =: \; \overline{X_\infty^*}(\bar P), 
\label{SOCPLlimits} \\
\lim_{t\rightarrow\infty} \E[Q_t^*] &=& \Big(1-\frac{\pi_1}{2\lambda D}\Big) \big( D - \overline{X_\infty^*}(\bar P) \big) \; =: \; \overline{Q_\infty^*}(\bar P).  \nonumber
\enq
\end{Theorem}

\vspace{3mm}
\noindent {\bf Interpretation and comments.} We concentrate on the long-term investment decision of the social planner. First, the social planner invests in centralised generation with the same decision rule as the firm. The investment level is not necessarily the same as the firm since it depends on the level of distributed generation. But the arbitrage is  guided by the same rule. The social planner satisfies the residual demand of the consumer with centralised energy. Thus, the main point is the arbitrage made by the social planner regarding distributed generation. Second, the long-term investment in distributed generation depends on the demand $D$. The long-term {distributed generation capacity} is a linear function of the long-term energy price $\bar P$, but we notice that the set of possible asymptotic Pareto optima does not cover the interval $(0,D)$ when varying $\bar P$. If the carbon tax is high enough compared to the distributed energy cost, there is a minimal long-term distributed capacity which will be installed even if the centralised energy price is null. Further, the Pareto optimal long-run level of distributed generation decreases linearly with distributed energy per unit cost $c$. Thus, for any given target  level of distributed energy, there is a distributed energy cost under which subsidies are no longer required to achieve an asymptotic Pareto optimum. This result somehow tempers the above cited conclusion of Green and Léautier (2015) \cite{Green15}. Moreover, the carbon emission limit prices $\pi_0$ and $\pi_1$ have no effect on the system~(\ref{SOCPLdefKL}) which gives the constant $K^{11}$ to $K^{22}$ because they appear divided by $D$ which is large. Thus, it is possible to analyse the effect of the absence of carbon taxes $\pi_0$ and $\pi_1$ by simply setting them to zero in the formula of the long-term investment in distributed energy and get:
\beq
\label{XbarInf}
\overline{X^*_{\infty}} \; = \;  \frac{ \theta \bar P - \frac{\rho c}{b}}{2 \sigma^2 \check K^{11}},
\enq
where $\check K^{11}$ is the value of $K^{11}$ of system~(\ref{SOCPLdefKL}) when $\pi_0 = \pi_1 = 0$.  Further, the investment in distributed generation by the social planner has the same expression as the investment by the consumer. But, now the social planner does not compare the price of the centralised energy with the price of the distributed energy. The social planner compares the cost of the transmission infrastructure ($\theta \bar P$) to the price of the distributed energy. The constant $\check K^{11}$ does not depend on $\theta$. Thus, the social planner invests in distributed energy only if $\theta > \frac{\rho c}{b \bar P}$. The cost of infrastructure should exceed the ratio of prices between distributed and centralised generations. Thus, the social planner sees the distributed energy as a substitute for transmission infrastructure. If the former condition is satisfied, the long-term investment in distributed energy will depend on the relative costs of the two energy sources because they are taken into account in the constant $K^{11}$. Unfortunately, this constant does not allow a simple expression. Nevertheless, one can note that in the case where the social planner is faced with a pure technology arbitrage problem to satisfy consumer demand, that is, there is no transmission infrastructure costs, the planner will not invest in distributed energy, whatever the relative costs of the two technologies. This unexpected result is not a consequence of the McKean-Vlasov term of variance penalisation. It holds true even if $\eta$ equals zero. This behaviour is explained by the effect both technologies have on the demand satisfaction constraint. In the long run, random distributed capacity always results in a positive term $\E\big[ \int_0^\infty e^{-\rho t} (D - X_t)^2 dt \big]$, whereas dispatchable centralised technology can cancel it out.  

\vspace{3mm}
\noindent {\bf Numerical illustration.} With the same parameter values as in Section~\ref{ssec:ltsg}, when there is no carbon tax, a long-run price of centralised energy of $\bar P$ equal to $114$~\euro/MWh is required to justify an investment by the social planner in distributed capacity. This price has to be put into perspective with the price paid for energy by French consumers which is about 80~\euro/MWh or with the present electricity price in the wholesale market which is presently close to 50~\euro/MWh.

\section{Looking for an equilibrium price} \label{secequil} 
\label{sec:Equilibres}

\setcounter{equation}{0}
\setcounter{Assumption}{0} \setcounter{Theorem}{0}
\setcounter{Proposition}{0} \setcounter{Corollary}{0}
\setcounter{Lemma}{0} \setcounter{Definition}{0}
\setcounter{Remark}{0}

Our chief goal is to endogenise the electricity price $P$ through equilibrium considerations which result from the interactions between the players in the centralised and distributed generation markets. As seen in the previous  section, the social planner can achieve a large set of long-run situations between centralised and distributed generations by choosing a regulated price of centralised energy. Our purpose is to analyse what situations can be naturally obtained by self-interested players exposed to a price signal. We address two types of equilibrium: Pareto and Stackelberg.  

\subsection{Asymptotic Pareto efficiency}
We recall the notations from the previous sections: $\hat X$ $=$ $\hat X(P)$ is the optimal distributed generation production with control $\hat\alpha$ $=$ $\hat\alpha(P)$ for the consumer, given the price $P$ for the centralised energy of the firm, $\hat Q$ $=$ $\hat Q(\hat X(P))$ is the optimal centralised generation capacity with optimal control $\hat\nu$ $=$ $\hat\nu(\hat X(P))$ for the firm given the optimal distributed generation capacity $\hat X$ of the consumer, and finally $(X^*,Q^*)$ $=$ $(X^*(P),Q^*(P))$ is the optimal pair of distributed/centralised generation capacities with control $(\alpha^*,\nu^*)$ $=$ $(\alpha^*(P),\nu^*(P))$  which is determined by the social planner given the price $P$ received by the firm. Thus, a Pareto optimum is a price process $P^{Par}$ such that the optimal installed capacities for the consumer and firm  coincide with the ones of the social planner, in other words, such that: 
\beq
\hat X_t(P^{Par}) &=& X_t^*(P^{Par}), \;\;\; \forall t \geq 0, \label{paretoX} \\
\hat Q_t(\hat X_t(P^{Par})) &=&  Q_t^*(P^{Par}), \;\;\;  \forall t \geq 0.  \label{paretoQ}
\enq
Further, the second relation on the optimal centralised production is always satisfied for any price process $P$. Indeed, by definition of the social planner's problem, given a fixed $P$, we have:
\beqs
V^S \; = \; J(\alpha^*,\nu^*) \; = \;  \inf_{\alpha\in\Ac,\nu\in\Vc} J(\alpha,\nu) 
& \leq & J(\alpha^*, \hat \nu(X^{\alpha^*})),
\enqs
which implies that $J_2(\nu^*;X^{\alpha^*})$ $\leq$ $J_2(\hat\nu(X^{\alpha^*});X^{\alpha^*})$. Since 
$\hat\nu(X^{\alpha^*})$ is the unique optimal control  of $\inf_{\nu\in\Vc} J_2(\nu;X^{\alpha^*})$, this shows that $\nu^*$ $=$ $\hat\nu(X^{\alpha^*})$, i.e. 
\beqs
\hat Q_t(\hat X(P)) &=& Q_t^*(P),  \;\;\;  \forall t \geq 0.  
\enqs
The first relation \reff{paretoX} on the optimal distributed generation is not satisfied in general for a given price $P$ (this is due to the fact that in general 
${\rm arg}\min_{\alpha\in\Ac} J(\alpha,\nu)$ $\neq$  ${\rm arg}\min_{\alpha\in\Ac} J_1(\alpha)$), and the existence and determination of such a Pareto optimum price $P^{Par}$ ensuring this relation is a challenging issue that we are unfortunately not able to solve. The complex formulae for $\hat X(P)$ and $X^*(P)$ in terms of  the price process $P$ seem intractable for finding a Pareto optimum satisfying the relation \reff{paretoX} at any time $t$, and instead, we propose a weaker idea of Pareto efficiency on the asymptotic (stationary) level. 
\begin{Definition}
An asymptotic Pareto optimum price is a constant $\bar P^{Par}$ such that 
\beq \label{PARdef}
\overline{\hat X_\infty}(\bar P^{Par}) &=& \overline{X_\infty^*}(\bar P^{Par}). 
\enq
An asymptotic Pareto optimum price $\bar P^{Par}$  is said to be admissible if
\beq \label{PARadmiss}
 \bar P^{Par} > 0, \;\;\;\;\;  \overline{X_\infty^*}(\bar P^{Par}) \in (0,D), & \mbox{ and } &  \overline{Q_\infty^*}(\bar P^{Par}) \in (0,\infty). 
\enq
\end{Definition}
Compared to the classical Pareto condition \reff{paretoX}, in \reff{PARdef} we require the weaker condition  that the  consumer's optimal asymptotic {average} capacity coincides with the social planner's optimal asymptotic {average} capacity. The admissibility conditions in \reff{PARadmiss} formulate on one hand  the natural requirement that an asymptotic Pareto optimum price should be positive. On the other hand, we also require that the asymptotic {average} stationary level of consumer's installed capacity should not exceed her demand since otherwise, it would mean that the consumer stably sells energy to the firm. The last admissibility condition states that the firm installed capacity should not be negative on the long-term. 

We provide existence and characterisation results for an admissible asymptotic Pareto optimum price. 
\begin{Theorem} \label{theoPar}
Let \eqref{rhosigma}, \eqref{integP} and \eqref{CondParamSocPl} hold. A necessary and sufficient condition for the existence of an admissible asymptotic Pareto optimum price  is that
  \begin{gather}
  \label{PARiff}
  \begin{cases}
  2 \sigma^2 K D \Big( 2 \lambda D \big( 2\pi_1 - \pi_0 - \frac{\rho c}{b} \big) - \pi_1^2 \Big) + \frac{\rho c}{b} \Big(4 \lambda D \big(\pi_1-\pi_0 + \sigma^2 K^{11} D \big)- \pi_1^2\Big) >0, \\
  -\pi_0 + 2 \sigma^2 D (K^{11} - \frac{\th}{1+\th} K) + \frac{1}{1+\th} \frac{\rho c}{b} \; > \; 0, \\
  2 \lambda D \big( 2\pi_1 - \pi_0 - \frac{\rho c}{b} \big) - \pi_1^2 + 2\lambda D \frac{\th}{1+\th} \frac{\rho c}{b} \; > \;  0,\\
  \pi_1 < 2 \lambda D,
  \end{cases}
  \nonumber \\
  \text{or}
  \\
  \begin{cases}
  2 \sigma^2 K D \Big(2 \lambda D \big( 2\pi_1 - \pi_0 - \frac{\rho c}{b} \big) - \pi_1^2 \Big) + \frac{\rho c}{b} \Big(4 \lambda D \big(\pi_1-\pi_0 + \sigma^2 K^{11} D \big)- \pi_1^2\Big) < 0, \\
  -\pi_0 + 2 \sigma^2 D (K^{11} - \frac{\th}{1+\th} K) +  \frac{1}{1+\th} \frac{\rho c}{b} \; < \;  0, \\
  2 \lambda D \big( 2\pi_1 - \pi_0 - \frac{\rho c}{b} \big) - \pi_1^2  + 2\lambda D \frac{\th}{1+\th} \frac{\rho c}{b} \; < \;  0,\\
  \pi_1 < 2 \lambda D,
  \nonumber
  \end{cases}
  \end{gather} 
where $K$ and $K^{11}$ are the constants defined in \reff{CONSdefKL} and \reff{SOCPLdefKL}. In this case, the unique admissible asymptotic Pareto optimum price is given by
\begin{equation}
\label{PARequil}
\bar P^{Par} = \frac{1}{1+\th}\frac{ 2 \sigma^2 K D \Big(2 \lambda D \big( 2\pi_1 - \pi_0 - \frac{\rho c}{b} \big) - \pi_1^2 \Big) 
+ \frac{\rho c}{b} \Big( 4 \lambda D \big(\pi_1-\pi_0 + \sigma^2 K^{11} D \big)- \pi_1^2 \Big)}{ 4 \lambda D \big(\pi_1-\pi_0 + \sigma^2 K^{11} D \big)- \pi_1^2 
- 4 \frac{\th}{1+\th} \lambda \sigma^2 K D^2 }.
\end{equation}
The corresponding equilibrium production is: 
\begin{equation*}
\overline{\hat X_\infty}(\bar P^{Par}) = \overline{X_\infty^*}(\bar P^{Par}) = D \frac{ 2 \lambda D \big( 2\pi_1 - \pi_0 - \frac{\rho c}{b}  \big) - \pi_1^2 
+ 2\lambda D \frac{\th}{1+\th} \frac{\rho c}{b}}{ 4 \lambda D \big(\pi_1-\pi_0 + \sigma^2 K^{11} D \big)- \pi_1^2 
- 4 \frac{\th}{1+\th} \lambda \sigma^2 K D^2}.
\end{equation*}
\end{Theorem}
\noindent {\bf Proof.} In order to simplify the notations, we set
\beqs
\chi_1 \; := \;   2 \lambda D \big( 2\pi_1 - \pi_0 - \rho c / b  \big) - \pi_1^2,
& & 
\chi_2 \; := \;   4 \lambda D \big(\pi_1-\pi_0 + \sigma^2 K^{11} D \big)- \pi_1^2.
\enqs
We get \eqref{PARequil} by solving the linear equation in \eqref{PARdef} from the expressions of the stationary production level for the consumer and social planner in \reff{CONSdefXinfty} and \reff{SOCPLlimits}. We now verify when the admissibility conditions \eqref{PARadmiss} are satisfied. We have $\bar P^{\text{Par}} >0$ if and only if
\begin{equation*}
\begin{cases}
2 \sigma^2 K D \chi_1 + \frac{\rho c}{b} \chi_2 >0, \\
\chi_2 - 4\frac{\th}{1+\th}\lambda \sigma^2 K D^2 >0,
\end{cases}
\qquad
\text{or}
\qquad
\begin{cases}
2 \sigma^2 K D \chi_1 + \frac{\rho c}{b} \chi_2 < 0, \\
\chi_2 - 4\frac{\th}{1+\th}\lambda \sigma^2 K D^2 < 0.
\end{cases}
\end{equation*} 
Moreover, we have $\overline{X_\infty^*}(\bar P^{Par})$ $\in$ $(0,D)$ if and only if
\begin{equation*}
0 < \frac{ \chi_1 + 2 \lambda D \frac{\th}{1+\th}\frac{\rho c}{b}}{ \chi_2 - 4\frac{\th}{1+\th}\lambda \sigma^2 K D^2} < 1,
\end{equation*}
that is, 
\begin{equation*}
\begin{cases}
\chi_2 - 4\frac{\th}{1+\th}\lambda \sigma^2 K D^2 > \chi_1 + 2 \lambda D \frac{\th}{1+\th} \frac{\rho c}{b}, \\
\chi_1 + 2 \lambda D \frac{\th}{1+\th}\frac{\rho c}{b} > 0,
\end{cases}
\quad
\text{or}
\quad
\begin{cases}
\chi_2 - 4\frac{\th}{1+\th}\lambda \sigma^2 K D^2 < \chi_1 + 2 \lambda D \frac{\th}{1+\th} \frac{\rho c}{b},\\
\chi_1 + 2 \lambda D \frac{\th}{1+\th} \frac{\rho c}{b} < 0.
\end{cases}
\end{equation*} 
Further, we have $\overline{Q_\infty^*}(\bar P^{Par})$ $>$ $0$ if and only if $\pi_1 < 2 \lambda D$, and we then obtain \eqref{PARiff} by combining the three sets of conditions.
\ep
 
\vspace{3mm}
\noindent {\bf Interpretation and comments.} The conditions in Theorem \ref{theoPar} are verified under reasonably weak conditions for the parameters. Indeed, the first inequality in the first system in \eqref{PARiff} is clearly satisfied when $\chi_1>0$ (notice that $\chi_2>0$ by \eqref{CondParamSocPl} and $K^{11}>0$): hence, a sufficient condition for the existence of an asymptotic Pareto optimum price is:
\begin{equation*}
\begin{cases}
2 \lambda D \big( 2\po - \pz - \frac{\rho c}{b} \big) - \po^2 >0, \\
-\pz + 2 \sigma^2 D (K^{11} - \frac{\theta}{1+\theta} K) + \frac{1}{1+\theta} \frac{\rho c}{b} >0, \\
\po < 2 \lambda D,
\end{cases}
\end{equation*}
which is satisfied under mild and practically reasonable assumptions on the coefficients, namely $\po>0$ and $D$ is big enough. In the asymptotic Pareto optimum, the demand level $D$ is large compared to other parameters and thus, it suggests to take the limit 
$D \rightarrow \infty$:
 \begin{gather*}
 \lim_{D \to \infty} \bar P^{\text{Par}} (D) = \frac{1}{1+\th}\frac{K(2 \pi_1 - \pi_0)  + \frac{\rho c}{b} \big(\check K^{11} -K\big)}{\check K^{11} - \frac{\th}{1+\th} K},
 \\
 \lim_{D \to \infty}\Big(\overline{X_\infty^*}(\bar P^{\text{Par}})\Big)(D) = \frac{1}{2\sigma^2} \frac{2 \po - \pz - \frac{1}{1+\th} \frac{\rho c}{b}}{\check K^{11} - \frac{\th}{1+\th} K},
 \end{gather*}
where $\check K = \{\check K^{ij}\}_{1 \leq i,j \leq 2}$ is a symmetric definite positive matrix solution to
\begin{equation*}
\begin{cases}
\frac{b^2}{\gamma}(\check K^{11})^2 + \frac{1}{h}(\check K^{12})^2 +\rho \check K^{11} - \sigma^2 \check K^{11} - \lambda - \eta = 0,
\\
\frac{b^2}{\gamma}\check K^{11} \check K^{12} + \frac{1}{h} \check K^{12} \check K^{22} +\rho \check K^{12} - \lambda = 0,
\\
\frac{b^2}{\gamma}(\check K^{12})^2 + \frac{1}{h}(\check K^{22})^2 +\rho \check K^{22} -\lambda = 0,
\end{cases}
\end{equation*}
The price of the asymptotic Pareto optimum is positive as soon as the carbon tax $\pi_1$ is high enough. The higher the carbon tax the higher the Pareto price is. This finding holds true also for the distributed energy cost $\rho c/b$ because $\check K^{11} -K$ is positive in practical numerical cases.
 
\vspace{3mm} 
\noindent {\bf Social effect of carbon tax.}  In the absence of  a carbon tax, that is,~$\pi_0,\pi_1=0$, we have
\beqs 
\bar P^{Par} &=&  \frac{\rho c}{b(1+\th)} \frac{\check K^{11}-K}{\check K^{11}- \frac{\th}{1+\th} K}, \\
\overline{X_\infty^*}(\bar P^{Par}) &=&  -\frac{\rho c}{2b\sigma^2}\frac{1}{\check K^{11}- \frac{\th}{1+\th} K}.
\enqs
In particular, we see that the limiting distributed production is always negative, and thus the asymptotic Pareto optimum price is not admissible. Since  a classical (in the sense of \reff{paretoX}) Pareto optimum price is obviously an asymptotic Pareto equilibrium, this means that there  cannot exist admissible Pareto optimum price if  $\pi_0=\pi_1=0$. In other words, the presence of a carbon tax  is necessary to get admissible Pareto efficiency. The reason for this result comes from the way infrastructure costs impact differently the consumer and the social planner's decision to invest in distributed energy. The consumer invests as soon as the total energy cost of centralised energy $(1+\theta) \bar P$ is higher than the cost of distributed energy $\rho c/b$ while the social planner only needs $\theta \bar P$ to be higher than the cost of distributed energy. It results that the consumer has a lower investment threshold price than the social planner, making agreements possible only for negative values of distributed generation. Adding carbon tax to centralised energy allows to decrease the investment threshold price of the social planner. If the {carbon tax} is large enough, then the social planner investment threshold price can become lower than the investment threshold of the consumer, opening the door for an equilibrium.

\subsection{Asymptotic Stackelberg equilibrium}

We consider here a Stackelberg equilibrium where the firm is viewed as a {\it leader} that offers a price to the consumer (the {\it follower}). The firm knows the optimal distributed generation $\hat X$ $=$ $\hat X(P)$, and searches for an equilibrium price that minimises its total cost in \reff{costenergy}:  $J_2(\nu;\hat X(P))$ $=$ $J_2(\nu,P;\hat X(P))$ where we now stress the dependence of the cost  on the price. The Stackelberg problem for the firm is then formulated as:
\beq 
\inf_{(\nu,P)} \E \Big[ \int_0^\infty e^{-\rho t} \Big( h \nu_t^2  & - &   P_t\big(D - \hat X_t(P) \big)   \label{pbsta} \\
& &   + \;\;\;\;\;   \pi^{\hat \alpha,\nu}_t \big(D - \hat X_t(P) \big)  + \lambda\big(D - \hat X_t(P) - Q_t^\nu \big)^2 \Big)dt\Big], \nonumber
\enq
and if an optimal control $(\check \nu, \check P)$ to this problem exists,  we say that $\check P$ is a Stackelberg equilibrium price. Unfortunately, given the expression of $\hat X(P)$ that can be obtained from Theorem \ref{theoconsu},  the problem \reff{pbsta} seems intractable, so we define an alternative and more tractable concept of Stackelberg equilibrium. 
 
Consider a stationary situation, where the price is a constant $p$, the centralised installed capacity is a constant $q$, and the consumer's distributed capacity is a constant $x$. Then, the cost for the firm is equal to:
\beqs
C_f(p,x,q) &=& \big(-  p + \frac{\pi_0}{D} x + \frac{\pi_1}{D} q \big)(D-x) + \lambda (D-x-q)^2.
\enqs
Because the situation is static, no installation costs appear. From Proposition \ref{propconsulimit} and Theorem \ref{theo2}, given a stationary price $\bar P$, the optimal stationary installed distributed capacity for the consumer is $\overline{\hat X_\infty}(\bar P)$ and the optimal stationary production for the firm is $\overline{\hat Q_\infty}(\bar P)$, where
\beqs
\overline{\hat X_\infty}(\bar P) \;=\; \frac{(1+\th) \bar P - \frac{\rho c}{b}}{2\sigma^2 K}, & & \qquad \overline{\hat Q_\infty}(\bar P) = \big(1 - \frac{\pi_1}{2\lambda D}\big)\big(D- \overline{\hat X_\infty}(\bar P)\big).
\enqs
\begin{Definition}
An asymptotic Stackelberg equilibrium price is a constant $\bar P^{Sta}$ such that 
\beqs
\bar P^{Sta} & \in & {\arg\min}_{\bar P \in\R}  C_f\big(\bar P, \overline{\hat X_\infty}(\bar P) , \overline{\hat Q_\infty}(\bar P) \big). 
\enqs
An asymptotic Stackelberg equilibrium price $\bar P^{Sta}$ is said to be admissible if
\beqs
\bar P^{Sta} > 0, \;\;\;\;\;  \overline{\hat X_\infty}(\bar P^{Sta}) \in (0,D), & \mbox{ and } &  \overline{\hat Q_\infty}(\bar P^{Sta}) \in (0,\infty).  
\enqs
\end{Definition}

Notice that, unlike the Pareto case, a Stackelberg equilibrium is not necessarily an asymptotic Stackelberg equilibrium. We provide existence and characterisation results for an admissible asymptotic Stackelberg equilibrium price.  
 
 \begin{Theorem}
Let \eqref{rhosigma} and \eqref{integP} hold. A necessary and sufficient condition for the existence of an admissible asymptotic Stackelberg equilibrium price is that:
  \begin{gather} \label{condadmi}
   \begin{cases}
    \sigma^2KD \pi_0 + (2\sigma^2KD + \frac{\rho c}{b}) \Big[ \frac{1}{1+\th}\sigma^2KD + \pi_1 - \pi_0 - \frac{\pi_1^2}{4\lambda D}\Big] >0, \\
    \frac{1}{1+\th}(2\sigma^2KD - \frac{\rho c}{b}) + 2\pi_1 - \pi_0 - \frac{\pi_1^2}{2\lambda D} >0, \\
    \frac{1}{1+\th}(2\sigma^2KD + \frac{\rho c}{b}) - \pi_0 > 0,\\
    \pi_1 < 2 \lambda D,
   \end{cases}
  \end{gather} 
where we recall that $K$ is the constant defined in \reff{CONSdefKL}. In this case, the equilibrium is unique and given by
\beqs
\bar P^{Sta} &=&  \frac{\sigma^2KD \pi_0 + (2\sigma^2KD + \frac{\rho c}{b}) \Big[\frac{1}{1+\th}\sigma^2KD + \pi_1 - \pi_0 -\frac{\pi_1^2}{4\lambda D}\Big]}
{(1+\th)\Big[\frac{1}{1+\th}2\sigma^2KD + \pi_1 - \pi_0 - \frac{\pi_1^2}{4\lambda D}\Big]}.
\enqs
 \end{Theorem}
\noindent {\bf Proof.} Observing that there is a one-to-one correspondence between $\bar P$ $\in$ $\R$ and $x$ $=$ $\overline{\hat X_\infty}(\bar P)$ $\in$ $\R$, we see that 
\beqs
\min_{\bar P\in\R} C_f\big(\bar P, \overline{\hat X_\infty}(\bar P) , \overline{\hat Q_\infty}(\bar P) \big) &=& \min_{x\in R} C_f(\bar P(x),x,\bar Q(x)),
\enqs
where $\bar P(x)$ is the inverse function of $\bar P$ $\mapsto$ $\overline{\hat X_\infty}(\bar P)$ given by $\bar P(x)$ $=$ 
$\frac{1}{1+\th} \big(- 2\sigma^2K (D-x) + 2\sigma^2KD + \frac{\rho c}{b}\big)$, and 
$\bar Q(x)$ $=$ $\overline{\hat Q_\infty}(\bar P(x))$  $=$ $(1-\frac{\pi_1}{2\lambda D})(D-x)$.  Hence, we see that
\beqs
C_F(x) &:=& C_f(\bar P(x),x,\bar Q(x)) \\
&=& \big[ -\frac{1}{1+\th}(2\sigma^2KD + \frac{\rho c}{b}) + \pi_0 \big](D-x) \\
& & \;\;\; + \; \bigg[\frac{1}{1+\th}2\sigma^2K - \frac{\pi_0}{D} + \frac{\pi_1}{D} - \frac{\pi_1^2}{4\lambda D^2} \bigg] (D - x)^2,
\enqs
and  a minimum for $C_F$ exists if and only if 
\begin{equation*}
 \chi := \frac{1}{1+\th}2\sigma^2KD + \pi_1 - \pi_0 - \frac{\pi_1^2}{4\lambda D} > 0.
\end{equation*}
Under such an assumption, the minimum is reached  at $x^{Sta}$ as defined by
\beqs
D- x^{Sta}  &=&  \frac{D}{2} \frac{\frac{1}{1+\th}(2\sigma^2KD + \frac{\rho c}{b}) - \pi_0}{\frac{1}{1+\th}2\sigma^2KD + \pi_1 - \pi_0 - \frac{\pi_1^2}{4\lambda D}},
\enqs
which corresponds to an asymptotic Stackelberg price
\beqs
\bar P^{Sta} &=& \bar P(x^{Sta}) \\
&=& \frac{\sigma^2KD \pi_0 + (2\sigma^2KD + \frac{\rho c}{b}) \Big[\frac{1}{1+\th}\sigma^2KD + \pi_1 - \pi_0 
- \frac{\pi_1^2}{4\lambda D}\Big]}{(1+\th)\Big[\frac{1}{1+\th}2\sigma^2KD + \pi_1 - \pi_0 - \frac{\pi_1^2}{4\lambda D}\Big]}.
\enqs
  
We now investigate the admissibility of such an asymptotic Stackelberg equilibrium price. First, under the condition $\chi$ $>$ $0$, we notice that 
$\bar P^{Sta}$ $>$ $0$ if and only if 
\beqs \label{first}
\sigma^2KD \pi_0 + (2\sigma^2KD + \frac{\rho c}{b}) \Big[\frac{1}{1+\th}\sigma^2KD + \pi_1 - \pi_0 - \frac{\pi_1^2}{4\lambda D}\Big] &>& 0.
\enqs
Secondly, we find that $\overline{\hat X_\infty}(\bar P^{Sta})$ $=$ $x^{Sta}$ $\in$ $(0,D)$ if and only if $D-x^{Sta}$ $\in$ $(0,D)$, i.e.
\begin{equation*}
   0 < \frac{\frac{1}{1+\th}(2\sigma^2KD + \frac{\rho c}{b}) - \pi_0}{ 2\chi} < 1,
  \end{equation*}
  that is (recall that $\chi >0$)
  \begin{equation*} \label{second}
   \begin{cases}
    2\chi > \frac{1}{1+\th}(2\sigma^2KD + \frac{\rho c}{b}) - \pi_0, \\
    \frac{1}{1+\th}(2\sigma^2KD + \frac{\rho c}{b}) - \pi_0 > 0.
   \end{cases}
  \end{equation*} 
Further, $\overline{\hat Q_\infty}(\bar P^{Sta})$ $=$ $\bar Q(x^{Sta})$ $>$ $0$ if and only if $\pi_1 < 2 \lambda D$.  By combining all these conditions, we obtain the existence of an admissible asymptotic Stackelberg equilibrium if and only if:
 \begin{equation*}
  \begin{cases}
   \sigma^2KD \pi_0 + (2\sigma^2KD + \frac{\rho c}{b}) \Big[\frac{1}{1+\th}\sigma^2KD + \pi_1 - \pi_0 - \frac{\pi_1^2}{4\lambda D}\Big] > 0, \\
   2\chi > \frac{1}{1+\th}(2\sigma^2KD + \frac{\rho c}{b}) - \pi_0, \\
   \frac{1}{1+\th}(2\sigma^2KD + \frac{\rho c}{b}) - \pi_0 > 0, \\
   \pi_1 < 2 \lambda D.
  \end{cases}
 \end{equation*} 
\ep

\noindent {\bf Interpretation and comments.} The admissibility conditions \reff{condadmi} are verified under reasonably weak conditions for the parameters, in particular that the demand $D$ is large enough. Thus, if the demand goes to infinity, the asymptotic Stackelberg equilibrium price tends to go to infinity, contrary to the asymptotic Pareto optimum which admits a finite limit. Thus, {for $D$ sufficiently large,} the Stackelberg long-term price should exceed the Pareto optimum. We use a numerical example to illustrate the discrepancy between the situation where the two players try to agree in a leader/follower game and the situation where the social planner tries to find a price that will suit both interest. 

\vspace{3mm}\noindent {\bf Numerical illustration.} Using the same parameter values as in sections \ref{ssec:ltsg}
and \ref{secsocial} and considering that the carbon tax limits are $\pi_0 = 0$, $\pi_1 = 100$ \euro/MWh,  the residual demand satisfaction constraints penalty is $\lambda = 100 \times 8760$ \euro/MW$^2$/year, and a total demand of $D = 50$~GW, we find an asymptotic Pareto optimum price according to formula \eqref{PARequil} of 91~\euro/MWh and a long-term distributed capacity of 1.1~GW. On the other hand, the asymptotic Stackelberg equilibrium yields a long-term distributed capacity of 25.7~GW and a  long-term price of 424~\euro/MWh. The Stacklelberg equilibrium results in a significant over-investment in distributed generation compared to the Pareto optimum. This over-investment can be explained by the high price charged by the firm for its centralised energy. Indeed, in the Stackelberg situation, the firm knows that whatever the consumer's investment in distributed generation, distributed generation will not always be enough to satisfy consumer's demand because the uncertainty of her generation is proportional to the installed capacity of distributed generation. When there is no sun, there is no generation whatever the number of installed solar panels. Thus, the firm can charge the consumer a high price each time consumer's demand requires centralised generation. Knowing this risk, the consumer tries to reduce her residual demand by investing in a high quantity of distributed generation while taking into account her own aversion to its variance. Multiplying by 10 the variance penalty parameter results only in a small reduction of distribution generation to $25.2$~GW but to a large increase in the asymptotic Stackelberg price to 1115~\euro/MWh. The more the consumer is averse to variance, the more the firm can charge her a high price for its dispatchable centralised generation.

\vspace{3mm}

\noindent {\bf Social effect of carbon tax.} We conclude this section by discussing the particular case when there is no carbon tax, that is, $\pi_0,\pi_1=0$. In this case, 
\beqs
\bar P^\text{Sta} &=& \frac{1}{1+\th}\big(\sigma^2 K D + \frac{\rho c}{2b}\big),   \\
\overline{\hat X_\infty}(\bar P^\text{Sta})  & = & \frac{D}{2} - \frac{\rho c}{4 b \sigma^2 K}, \;\; \mbox{ i.e. } \;\; D - \overline{\hat X_\infty}(\bar P^\text{Sta}) \; = \; \frac{D}{2} + \frac{\rho c}{4 b \sigma^2 K},
\enqs
and the admissibility conditions reads as $2\sigma^2KD > \rho c/b$ and $\pi_1 < 2 \lambda D$. Further, we make some observations. 
\begin{itemize}
\item[-] Provided that $D$ is big enough, we have an admissible Stackelberg equilibrium (unlike the Pareto case, where we had no admissible equilibria in the case $\pi_0=\pi_1=0$). In particular, the price is always positive. 
\item[-] The equilibrium price is an affine function of the demand. Thus, for $D$ sufficiently large, the asymptotic Stackelberg price will be greater than the Pareto one. 
\item[-] The equilibrium production for the consumer is always smaller than $D/2$ whatever the installation cost $c$ of distributed energy.  The company has clearly an interest in limiting the amount produced by the consumer, but it is remarkable to have a simple and explicit upper bound: $D/2$.
\end{itemize}

\section{Extensions}
\label{sec:Extensions}
We now consider an extension of the model proposed for the consumer and the social planner. In the previous sections, the consumer's installation costs $C_1(\alpha_t)$ $=$ $c\h_t + \gamma \h_t^2$ only consider the current choice $\h_t$. Here, we consider a learning-by-doing effect on the distributed generation per unit cost. The cumulative installed distributed capacity, that is, $\int_0^t \h_s ds$,  drives the price of solar panels down, which is typically observed for renewable energy technologies (see \cite{Rubin15}). This suggests the following new definition for the installation costs: 
\[
\tilde C_1(\alpha) = c\h_t + \gamma \h_t^2 - \tilde \mu \h_t \eee\big[\textstyle\int_0^t \h_s ds \big],
\]
 for some $\tilde\mu$ $>$ $0$. We consider a linear impact form for the learning-by-doing effect instead of the exponential form for sake of tractability.
Moreover, to reduce the number of variables involved, we notice that $dX^\h_t = b \h_t dt + \sigma X^\h_t dW_t$ implies $\eee[\int_0^t \h_s ds]  = (\eee[X^\h_t] \!-\! x_0)/b$, so that we can rewrite $\tilde C_1(\alpha)$ as 
$c\h_t + \gamma \h_t^2 - \mu \h_t \eee[X^\h_t]$ (we have set $\mu = \tilde \mu /b$). 
\begin{Remark}
{\rm 
Another interpretation of the present approach is possible. Assume that $N$ consumers are present in the market. It is reasonable to assume that the installation costs for each consumer depend on the other players' behaviour: if the consumers are on average installing new panels, the costs are smaller. We assume this decrease is quadratic with respect to the mean of the produced electricity; in other words, the installation costs for consumer $i \in \{1,\dots,N\}$ are modelled by $\gamma (\h^i_t)^2 + c \h^i_t - \mu \h^i_t (\frac{1}{N} \sum_{j=1}^{N}X^{\h^j}_t)$. As $N \to \infty$ (aggregated consumers), the expression above is approximated by $\gamma \h^2_t + c \h_t - \mu \h_t \eee[X^\h_t]$, which is the model studied in this section.
} 
\ep
\end{Remark}

We are then led to consider the following generalisation of the consumer's problem:
 \begin{equation*}
 \inf_{\h} \eee \left[ \int_0^\infty e^{-\rho t}\bigg( c\h_t + \gamma \h_t^2 - \mu \h_t\eee[X^\h_t] + (1+\th)P_t \big(D - X^\h_t) + \eta \text{Var}[X^\h_t] \bigg)dt\right].
 \end{equation*}
This problem can be solved by the same technique as above. The results for the optimal control and production are as follows.
  
 \begin{Theorem} 
Let \eqref{rhosigma} and \eqref{integP} hold and assume that $\mu$ is small enough to satisfy 
 	\begin{equation}
 	\label{CondParamConsExt}
 	\sigma^2K - \frac{\rho\mu}{2b} >0,
 	\end{equation}
 	with $K$ as in Theorem \ref{theoconsu}, and set 
 	\begin{equation}
 	\label{mu1} 
 	\tilde \Lambda = \gamma \frac{-\rho + \sqrt{\rho^2 + \frac{4b^2}{\gamma}(\sigma^2K - \frac{\rho\mu}{2b})}}{2b^2}.  
 	\end{equation}
 The optimal solar panels installation rate $\hat\alpha_t^{ext}$ and the corresponding optimal cumulative production $\E[\hat X_t^{ext}]$ are given by the same formulas as in Theorem \ref{theoconsu}, with $\Lambda$ replaced by $\tilde \Lambda$.  In particular, if there exists $\bar P$ $>$ $0$ s.t.~$\lim_{t\rightarrow\infty}\E[P_t]$ $=$ $\bar P$, the optimal cumulative production of distributed generation admits a stationary level:
  \beq
  \label{}
  \lim_{t\rightarrow\infty} \E[\hat X_t^{ext}] \; = \;  \frac{ (1+\th) \bar P - \frac{\rho c}{b}}{2 \sigma^2 K - \frac{\rho \mu}{b}} \; =: \; \overline{\hat X_\infty^{ext}}(\bar P),  
  & \mbox{ and so } & 
  \lim_{t\rightarrow\infty} \E[\hat \alpha_t^{ext}] \; = \;  0.
  \enq
 \end{Theorem}
 
\vspace{3mm} 

\noindent{\bf Interpretation and comments.} The learning-by-doing parameter increases the value of the distributed generation limit, which is reasonable: the installation costs are smaller, so that the limit production is bigger. However, we find that the effect of $\mu$ on the asymptotic {average} distributed capacity is quite small. In our context, a crude but reasonable way of determining a value of $\mu$ is to consider that if all the demand is satisfied with distributed generation, then  the initial cost $c$ would be divided by a factor four. Using the same data as the numerical illustration in Section~\ref{ssec:ltsg}, this reasoning would result in a value for $\mu$ $=$ $ 3/4 \times c/D\approx 45$ which is negligible before $2 \sigma^2 K \approx 1.1 \, 10^3$.

\vspace{5mm}

The social planner's problem now reads:
\begin{multline}
\inf_{(\h,\nu)} \E \Big[ \int_0^\infty e^{-\rho t} \Big( c\h_t + \gamma \h_t^2 - \mu \h_t\eee[X^\h_t] + h \nu_t^2 
\\
+ (\th P_t + \pi_t^{\alpha,\nu}) (D- X_t^\alpha) + \eta {\rm Var}[X_t^\alpha] +  \lambda\big(D - X^\alpha_t-Q_t^\nu\big)^2  \Big) \Big].  
\end{multline} 
The optimal controls are as follows.
 
\begin{Theorem} 
Let \eqref{rhosigma} and \eqref{integP} hold and assume that
\begin{equation}
\label{CondParamSocPlExt}
D > \max \Big\{\frac{\pz}{\lambda}; \frac{\po^2}{4\lambda(\po-\pz)}\Big\}
\qquad\text{and}\qquad
\sigma^2 K^{11} - \frac{\rho \mu}{2 b} >0,
\end{equation}
with $K^{11}$ as in Theorem \ref{theosocial}. Let $\{\tilde\Lambda^{ij} \}_{1\leq i,j \leq 2}$ be the unique symmetric positive-definite solution to the stochastic algebraic Riccati (SARE) system
\begin{equation}
\label{mu2}  
\begin{cases}
\frac{b^2}{\gamma}(\tilde\Lambda^{11})^2 + \frac{1}{h}(\tilde\Lambda^{12})^2 +\rho \tilde\Lambda^{11} - \sigma^2 K^{11} - \lambda + \frac{\pi_0}{D} + \frac{\rho\mu}{2b}= 0,
\\
\frac{b^2}{\gamma}\tilde\Lambda^{11}\tilde\Lambda^{12} + \frac{1}{h}\tilde\Lambda^{12}\tilde\Lambda^{22} +\rho \tilde\Lambda^{12} - \lambda + \frac{\pi_1}{2D} = 0,
\\
\frac{b^2}{\gamma}(\tilde\Lambda^{12})^2 + \frac{1}{h}(\tilde\Lambda^{22})^2 +\rho \tilde\Lambda^{22} -\lambda = 0.
\end{cases}
\end{equation}
The optimal installation rates $\alpha_t^{*,ext}, u_t^{*,ext}$ and the corresponding optimal cumulative production $\E[X_t^{*,ext}], \E[Q_t^{*,ext}]$ are given by the same formulae as in Theorem \ref{theosocial}, with the matrix $\Lambda$ replaced by $\tilde \Lambda$. In particular, if there exists $\bar P$ $>$ $0$ s.t.~$\lim_{t\rightarrow\infty}\E[P_t]$ $=$ $\bar P$, the optimal cumulative production of distributed generation admits a stationary level:
  \beq 
  \lim_{t\rightarrow\infty} \E[X_t^{*,ext}] &=& \frac{2 \lambda D^2 \big( 2 \po - \pz + \th \bar P- \frac{\rho c}{b} \big) - D\po^2 }{4 \lambda D \big( \po - \pz +\sigma^2K^{11}D - \frac{\rho\mu D}{2b} \big) - \po^2} \; =: \; \overline{X_\infty^{*,ext}}(\bar P), 
  \\
  \lim_{t\rightarrow\infty} \E[Q_t^{*,ext}] &=& \Big(1-\frac{\pi_1}{2\lambda D}\Big) \big( D -  \overline{X_\infty^{*,ext}}(\bar P) \big) \; =: \; \overline{Q_\infty^{*,ext}}(\bar P).  \nonumber
  \enq
 \end{Theorem}
 
\vspace{5mm}

\noindent{\bf Interpretation and comments.} The impact of the learning-by doing effect appears in the asymptotic {average} distributed capacity with the correcting term $\rho \mu D/(2 b)$. With the parameter value considered for $\mu$, this term has a small impact on the asymptotic Pareto optimum distributed capacity.
  
\vspace{5mm}

The definition of asymptotic Pareto optimum prices is the same; the characterisation here reads as follows.

\vspace{2mm} 
 
 \begin{Theorem}
  Let \eqref{rhosigma}, \eqref{integP}, \eqref{CondParamConsExt} and \eqref{CondParamSocPlExt} hold. A necessary and sufficient condition for the existence of an admissible asymptotic Pareto optimum price  is that
  \begin{gather}
  \begin{cases}
  D \big(2\sigma^2 K - \frac{\rho \mu}{b}\big) \Big(  2\lambda D \big( 2 \po - \pz - \frac{\rho c}{b} \big) - \po^2  \Big) 
  \\ \hspace{3cm} +  \frac{\rho c}{b} \Big( 4 \lambda D \big( \po - \pz +\sigma^2K^{11}D - \frac{\rho\mu D}{2b} \big) - \po^2 \Big) >0, \\
  -\pz + 2 \sigma^2 D (K^{11} - \frac{\th}{1+\th} K ) + \frac{1}{1+\th} \frac{\rho c}{b} + \frac{\th}{1+\th} \frac{\rho \mu D}{b}>0, \\
  2\lambda D \big( 2 \po - \pz - \frac{\rho c}{b} \big) - \po^2 + 4 \lambda D \frac{\th}{1+\th} \frac{\rho c}{b} > 0,\\
  \pi_1 < 2 \lambda D,
  \end{cases}
  \nonumber \\
  \text{or}
  \\
  \begin{cases}
  D \big(2\sigma^2 K - \frac{\rho \mu}{b}\big) \Big(  2\lambda D \big( 2 \po - \pz - \frac{\rho c}{b} \big) - \po^2  \Big) 
  \\ \hspace{3cm} +  \frac{\rho c}{b} \Big( 4 \lambda D \big( \po - \pz +\sigma^2K^{11}D - \frac{\rho\mu D}{2b} \big) - \po^2 \Big) < 0, \\
  -\pz + 2 \sigma^2 D (K^{11} - \frac{\th}{1+\th} K ) + \frac{1}{1+\th} \frac{\rho c}{b} + \frac{\th}{1+\th} \frac{\rho \mu D}{b} < 0, \\
  2\lambda D \big( 2 \po - \pz - \frac{\rho c}{b} \big) - \po^2 + 4 \lambda D \frac{\th}{1+\th} \frac{\rho c}{b} < 0,\\
  \pi_1 < 2 \lambda D,
  \nonumber
  \end{cases}
  \end{gather} 
  where $K$ and $K^{11}$ are the constants defined in \reff{CONSdefKL} and \reff{SOCPLdefKL}. In this case, the unique admissible asymptotic Pareto optimum price is given by
  \begin{equation}
  \bar P^{Par} = \frac{1}{1+\theta}\frac{ D \big(2\sigma^2 K - \frac{\rho \mu}{b}\big) \Big(  2 \lambda D \big( 2 \po - \pz - \frac{\rho c}{b} \big) - \po^2 \Big) +  \frac{\rho c}{b} \Big( 4 \lambda D \big( \po - \pz +\sigma^2K^{11}D - \frac{\rho\mu D}{2b} \big) - \po^2 \Big) }{ 4 \lambda D \big( \po - \pz +\sigma^2K^{11}D - \frac{\rho\mu D}{2b} \big) - \po^2 - 2\lambda D^2 \frac{\th}{1+\th} \big(2 \sigma^2 K - \frac{\rho \mu}{b} \big) }.
  \end{equation}
  The corresponding equilibrium production is 
  \begin{equation*}
  \overline{\hat X_\infty}(\bar P^{Par}) = D \frac{ 2 \lambda D \big( 2 \po - \pz - \frac{\rho c}{b} \big) - \po^2 + 2 \lambda D \frac{\th}{1+\th} \frac{\rho c}{b}}{ 4 \lambda D \big( \po - \pz +\sigma^2K^{11}D - \frac{\rho\mu D}{2b} \big) - \po^2 - 2\lambda D^2\frac{\th}{1+\th} \big(2 \sigma^2 K - \frac{\rho \mu}{b} \big) }.
  \end{equation*}
 \end{Theorem}
 
\vspace{3mm}
\noindent{\bf Interpretation and comments.} Notice that the asymptotic Pareto price is barely impacted by the learning-by-doing parameter $\mu$.  
Further, the results do not change for the asymptotic Stackelberg equilibrium, as we have only changed the installation costs, which do not play any role in the definition of the Stackelberg equilibrium.

\section{Conclusions} \label{sec:conclu}
In this paper, we study  the interaction between electricity producers and consumers induced by the development of affordable distributed energy sources. Our model yet simple uses recent advances in the resolution of McKean-Vlasov optimal control problems with random coefficients. We show that in the long-term, the {\em laissez-faire} strategy leads to an over-investment in distributed energy sources compared to the socially desired outcome and to a higher price for the centralised energy. This effect can be understood as the possibility for power producers to charge at high prices the insurance they offer to the consumers for their dispatchable generation. Moreover, we show the crucial role of a significant carbon tax to justify the investment in distributed energy for the social planner. 
Although robust with respect to the energy price modelling, these results can be questioned regarding the reduction in the interaction between producers and consumers with distributed generation to an interaction between only two players. Indeed, competition between producers together with the possibility of energy exchanges between consumers might reduce the market power of electricity producers and help achieve a more socially and balanced equilibrium between them. These two points are left for further research.

\setcounter{equation}{0}
\setcounter{Assumption}{0} \setcounter{Theorem}{0}
\setcounter{Proposition}{0} \setcounter{Corollary}{0}
\setcounter{Lemma}{0} \setcounter{Definition}{0}
\setcounter{Remark}{0}

\appendix

\section{Appendix}

\setcounter{equation}{0} \setcounter{Assumption}{0}
\setcounter{Theorem}{0} \setcounter{Proposition}{0}
\setcounter{Corollary}{0} \setcounter{Lemma}{0}
\setcounter{Definition}{0} \setcounter{Remark}{0}
 
 \subsection{Proof of Theorem \ref{theoconsu}}
 \label{AppA1}
Given a candidate in the form of \eqref{CONScandidate}, our goal is to set the coefficients $K,\Lambda,Y,R$ so as to satisfy the condition in \eqref{conDc}.
By applying the It\^o formula and completing the square with respect to the control, we get \eqref{CONSito}, where the coefficients are defined by
\beqs
F(k, \dot k) &=& \dot k - \frac{b^2k^2}{\gamma} + (\sigma^2 - \rho)k + \eta, \\
G(k,\lambda, \dot \lambda) &=& \dot \lambda -\frac{b^2 \lambda^2}{\gamma} - \rho \lambda +\sigma^2k, \\
H_t(k, \lambda, y, \bar y, \dot y) &=& \dot y - \rho y  - \frac{b^2 k}{\gamma} (y - \bar y) - \frac{b \lambda}{\gamma} (c+ b \bar y) - (1+\th) P_t, \\
M(y, r, \dot r) &=& \dot r - \rho r  - \frac{1}{4\gamma} (c+ b y)^2, \\
A(x, \bar x, k , \lambda, y) &=& - \frac{b k}{\gamma}(x - \bar x) - \frac{b \lambda}{\gamma} \bar x - \frac{b}{2\gamma} y - \frac{c}{2\gamma}.
\enqs
Condition \eqref{CONSzerocoeff} then leads to the following equations for the coefficients:
\beqs
dK_t &=& \Big( \frac{b^2}{\gamma}K_t^2 + (\rho - \sigma^2) K_t -\eta \Big) dt,  \\
d\Lambda_t &=& \Big(\frac{b^2}{\gamma} \Lambda_t^2 +\rho \Lambda_t  - \sigma^2 K_t \Big) dt, \\
dY_t &=&  \Big( \rho Y_t + \frac{b^2 K_t}{\gamma} (Y_t - \eee[Y_t]) + \frac{b \Lambda_t}{\gamma} (c+ b \eee[Y_t]) + (1+\th) P_t \Big) dt + Z^Y_t dW^0_t, \\
dR_t &=& \Big( \rho R_t  + \frac{1}{4\gamma} (c+ b Y_t)^2 \Big) dt + Z^R_t dW^0_t.
\enqs
The first and second equations admit constant strictly positive solutions $K_t \equiv K$ and $\Lambda_t \equiv \Lambda$, with $K,\Lambda$ as in \eqref{CONSdefKL}. The third and the fourth equations are linear Backward Stochastic Diffe\-rential Equations (BSDEs), with the immediate solution 
(for the coefficient $Y_t$, first compute $Y_t - \eee[Y_t]$) given by:
\beq
Y_t &=& - (1+\th) \int_t^\infty e^{-(\rho + b^2K/\gamma)(s-t)}  \eee[ P_s| \mathcal{F}^0_t] ds \nonumber \\
& & \;\;\;  - \; (1+\th) \int_t^\infty \Big( e^{-(\rho + b^2\Lambda/\gamma)(s-t)} - e^{-(\rho + b^2K/\gamma)(s-t)} \Big) \eee[P_s] ds - \frac{bc\Lambda}{\rho\gamma + b^2\Lambda}, \;\;\;\;\;\;
\label{CONSdefYR} \\
R_t &=& - \frac{1}{4\gamma} \int_t^\infty e^{-\rho (s-t)} \eee[(bY_s + c)^2| \mathcal{F}^0_t] ds. \nonumber 
\enq
Furthermore, $Y_t$ is well-defined by the Holder inequality and \eqref{integP}, whereas \eqref{CONScond(i)} will imply that $R_t<\infty$. We get \eqref{CONScontropt} by plugging \eqref{CONSdefYR} into \eqref{CONSimplicitcontr}, while \eqref{CONSmeanprod} immediately follows by \eqref{dynX} and \eqref{CONScontropt}. By the procedure above, conditions (ii) and (iii) in Lemma \ref{lemverif} are verified (we will later prove that $\hat\alpha$ actually lies in $\Ac$). As for condition (i), we have to prove that 
$$
\lim_{t \to \infty} \E \big[ e^{-\rho t} (|Y_t|^2 + |R_t|)\big] \to 0.
$$
It immediately follows from the definition that $\E [ e^{-\rho t} |R_t|] \to 0$ as $t \to \infty$; as for the process $Y$, we prove the stronger condition
\beq
\label{CONScond(i)}
\E \Big[ \int_0^\infty e^{-\rho t} |Y_t|^2 dt  \Big] &<&  \infty.
\enq
By Jensen's inequality, Fubini's theorem, the law of iterated conditional expectations, and \eqref{integP}, we have: 
\beqs
& & \E \Big[  \int_0^\infty e^{-\rho t} \left( \int_t^\infty e^{-(\rho + b^2 K / \gamma)(s-t)} \E[P_s|\Fc_t^0] ds \right)^2 dt \Big] \\
& \leq &  \frac{1}{\rho + b^2 K / \gamma} \int_0^\infty \int_t^\infty e^{-\rho t} e^{-(\rho + b^2 K / \gamma)(s-t)}  \E[|P_s|^2] ds dt  \\
& \leq &  \frac{1}{(\rho + b^2 K / \gamma)b^2 K / \gamma} \int_0^\infty e^{-\rho s} \E[|P_s|^2] ds \; < \;  \infty.
\enqs
We deal with the other terms in $Y_t$ by the same arguments and obtain the required integrability condition \eqref{CONScond(i)}. Let us finally show  that $\hat\alpha \in \Ac$, i.e.~$\int_0^\infty e^{-\rho t}\E[|\hat\alpha_t|^2] dt$ $<$ $\infty$. Recall that $\hat\alpha$ is written as
\beqs \label{CONSdefG}
\hat\alpha_t & = &  - \frac{b K}{\gamma}\hat X_t + U_t,
\enqs
with
\beqs
U_t &:=& \frac{b (K - \Lambda)}{\gamma} \eee[\hat X_t] - \frac{b}{2\gamma} Y_t - \frac{c}{2\gamma}.
\enqs
By It\^o's formula to \reff{dynX} for $\alpha$ $=$ $\hat\alpha$, Young's inequality, and Gronwall's lemma, we have: 
\beqs
\E[ e^{-\rho t} |\hat X_t|^2 ] & \leq & \Big(  x_0^2 + \frac{b}{\epsilon} \int_0^t e^{-\rho s} \E[|U_s|^2] ds  \Big) e^{-(\rho - \sigma^2 - b\epsilon)t}, \;\;\;  t \geq 0, 
\enqs
for all $\epsilon$ $>$ $0$. For $\epsilon$ small enough, and under \reff{rhosigma}, 
we have $\rho - \sigma^2 - b\epsilon$ $>$ $0$, and thus it suffices to show that 
$\int_0^\infty e^{-\rho t} \E[|U_t|^2] dt$ $<$ $\infty$ to ensure the square-integrability condition:  $\int_0^\infty e^{-\rho t} \E[|\hat X_t|^2] dt$ $<$ $\infty$, and so 
$\hat\alpha$ $\in$ $\Ac$.  From the integrability condition  \eqref{CONScond(i)} on $Y$, we then have to check that  $\int_0^\infty e^{-\rho t}
\big(\E[\hat X_t]\big)^2 dt$ $<$ $\infty$, hence 
by  \eqref{CONSmeanprod} that (we set $B$ $=$ $b^2\Lambda /\gamma $)
\beqs
I &:=&  \int_0^\infty e^{-\rho t} \bigg( e^{-Bt} \int_0^t e^{Bs} \int_s^\infty e^{-(\rho + B)(u-s)} \E[P_u] du \, ds \bigg)^2 dt \; < \;  \infty.
\enqs 
By Fubini-Tonelli we can rewrite $I$ as: 
\beqs
I &=&  \frac{1}{(\rho+2B)^2} \int_0^\infty e^{-(\rho + 2B) t} \bigg(\int_0^\infty  \big(e^{(\rho + 2B) \min\{t,u\}}-1\big) e^{-(\rho + B)u} \E[P_u] du \bigg)^2 dt;
\enqs
by writing $e^{-(\rho + B)u}$ $=$ $e^{-2\epsilon u}e^{-(\rho + B -2\epsilon)u}$ for $\epsilon>0$, by Jensen's inequality w.r.t.~the measure $e^{-2\epsilon u} du$ and by Fubini-Tonelli, we get:
\beqs
I &\leq& \frac{1}{\epsilon(\rho+2B)^2} \int_0^\infty \left( \int_0^\infty e^{-(\rho + 2B)(t+u-2\min\{t,u\})} dt \right) e^{-(\rho-2\epsilon)u}\eee[P_u^2] du,
\enqs
which finally leads to
\begin{equation*}
I \leq \frac{2}{\epsilon(\rho + 2B)^3} \int_0^\infty e^{-(\rho- 2\epsilon)u}\eee[P_u^2] du.
\end{equation*}
Hence, from \eqref{integP} and for $\epsilon$ small enough we have $I$ $<$ $\infty$ and then $\hat\alpha$ $\in$ $\Ac$.
\ep

 \subsection{Proof of Theorem \ref{theo2}}  
 \label{AppA2}
 
 A suitable adaptation of Lemma \ref{lemverif} holds. We look for a candidate in the form of: 
 \beq
 \label{FIRMcandidate}
 v_t(q) &=& K_t q^2 + Y_t q + R_t, 
 \enq
 where the dynamics of the coefficients $K,Y,R$ are given by
 \beqs
 dK_t \; = \; \dot K_t dt, \qquad
 dY_t \; = \; \dot Y_t dt + Z_t^Y dW_t, \qquad
 dR_t \; = \; \dot R_t dt + Z_t^R dW_t,
 \enqs
 for some deterministic process $\dot K$ and $\F^0$-adapted processes $\dot Y$, $\dot R$, $Z^Y$, $Z^R$. As in Appendix \ref{AppA1}, we have assumed the quadratic coefficient to be deterministic. Moreover, since the randomness comes from the $\F$-adapted process $X^\h$, the stochastic part in $Y,R$ only depends on $W$. 
 
 By applying It\^o's formula to $S_t^\nu$ $=$ $e^{-\rho t} v_t(Q_t^\nu)$ $+$ $\int_0^t e^{-\rho s} \ell_s(Q^\nu_s,\nu_s)ds$ and completing the square, we get $d\E[S_t^\nu] = e^{-\rho t} \E[ \Dc_t^\nu ] dt$, with
 \beqs
 \E[\mathcal{D}^\nu_t] &=& \E\Big[ h \big(\nu_t - A(Q^\nu_y, K_t, Y_t) \big)^2+ F(K_t, \dot K_t) (Q^\nu_t)^2 
 + H_t(K_t,Y_t,\dot Y_t) Q^\nu_t + M(Y_t, R_t, \dot R_t) \Big],
 \enqs
 where the coefficients are defined by
 \beqs
 F(k, \dot k) &=& \dot k - \frac{k^2}{h} - \rho k + \lambda,  \\
 H_t(k, y, \dot y) &=& \dot y - \big( \rho + \frac{k}{h} \big) y - (2\lambda-\pi_1/D)\big(D - X^\h_t\big),  \\
 M(y, r, \dot r) &=& \dot r - \rho r  - \frac{y^2}{4h}, \\
 A(q, k, y) &=&  - \frac{k}{h} q - \frac{y}{2h}.
 \enqs
 Setting $\eee[\mathcal{D}^\nu_t] \geq 0$ for each $\nu$ and $\eee[\mathcal{D}^\nu_t] = 0$ for $\nu = \hat \nu$ provides the optimal control
 \beq  \label{FIRMimplicit}
 \hat \nu_t &=& A(Q^{\hat\nu}_t, K_t, Y_t)
 \enq
 and the following equations for the coefficients:
 \beqs
 dK_t &=& \bigg( \frac{K_t^2}{h} + \rho K_t - \lambda \bigg) dt, \\
 dY_t &=& \bigg( \Big(\rho + \frac{K_t}{h}\Big) Y_t + 2\lambda \Big(1-\frac{\pi_1}{2\lambda D}\Big)\big(D - X^\h_t\big) \bigg) dt + Z^Y_t dW_t, \\
 dR_t &=& \bigg( \rho R_t + \frac{1}{4h}Y^2_t \bigg) dt + Z^R_t dW_t.
 \enqs
 The first equation admits a constant solution $K_t \equiv \tilde K$, with $\tilde K$ as in \eqref{FIRMtildeK}, while the second and third equations are linear BSDEs, with immediate solutions:
 \begin{equation*}
 \label{FIRMdefYR}
 \begin{aligned} 
 &Y_t = - 2\lambda\Big(1-\frac{\pi_1}{2\lambda D}\Big) \int_t^\infty e^{-(\rho + \tilde K/h)(s-t)}  \eee[ D - X^\h_s| \mathcal{F}_t] ds, \\
 &R_t = - \frac{1}{4h} \int_t^\infty e^{-\rho (s-t)} \eee[Y_s^2| \mathcal{F}_t] ds.
 \end{aligned}
 \end{equation*}
 We get \eqref{FIRMoptcontr} by \eqref{FIRMimplicit} and \eqref{FIRMdefYR}, whereas \eqref{FIRMexpect} immediately follows by \eqref{FIRMoptcontr} and \eqref{dynQ}. Moreover, the conditions (i) and (iii) in Lemma \ref{lemverif} are verified by standard estimates as in Appendix \ref{AppA1}. Finally, we get the limit result by the same computations as the ones in Proposition \ref{prop:limitscons}.
\ep

 \subsection{Proof of Theorem \ref{theosocial}}  
 \label{AppA3}
 
  A suitable adaptation of Lemma \ref{lemverif} holds.   We use vector notations and a small change of variable by setting:
  \begin{gather*}
  \delta_t = 
  \begin{pmatrix}
  b \h_t \\ \nu_t
  \end{pmatrix},
  \qquad
  Z^\delta_t = 
  \begin{pmatrix}
  X^\h_t \\ Q^\nu_t
  \end{pmatrix},
  \end{gather*}
so that the dynamics of $Z^\delta$ are written as 
  \begin{equation}
  \label{dynZ}
  dZ^\delta_t = \delta_t dt + S Z^\delta_t dW_t,
  \qquad\qquad
  S=
  \begin{pmatrix}
  \sigma & 0 \\
  0 & 0 \\
  \end{pmatrix}.
  \end{equation}
  The payoff is rewritten as $\tilde g_t(Z^\delta_t, \eee[Z^\delta_t], \delta_t)$, where
  \begin{equation*}
  \tilde g_t(z, \bar z, d) = (z- \bar z)'Q(z-\bar z) + \bar z' (Q +\tilde Q) \bar z + T_t' z + d'N d + d'U,
  \end{equation*}
  where $'$ denotes transposition and the coefficients are defined by (notice that $Q,\tilde Q, N, U$ are constant, $T$ is stochastic and $\mathbb{F}^0$-adapted)
  \begin{gather*}
  Q=
  \begin{pmatrix}
  \lambda+\eta-\frac{\pi_0}{D} & \lambda-\frac{\pi_1}{2D} \\
  \lambda-\frac{\pi_1}{2D} & \lambda \\
  \end{pmatrix},
  \qquad
  \tilde Q=
  \begin{pmatrix}
  -\eta & 0 \\
  0 & 0 \\
  \end{pmatrix},
  \\
  T_t=
  \begin{pmatrix}
  -2\lambda D - \th P_t + \pi_0 \\
  -2\lambda D + \pi_1\\
  \end{pmatrix},
  \qquad
  N=
  \begin{pmatrix}
  \frac{\gamma}{b^2} & 0 \\
  0 & h \\
  \end{pmatrix},
  \qquad
  U=
  \begin{pmatrix}
  \frac{c}{b} \\
  0 \\
  \end{pmatrix}.
  \end{gather*}
  Correspondingly, we consider candidates in the form of:
  \begin{equation}
  \label{valfunc}
  v_t(z,\bar z) = (z-\bar z)' K_t (z-\bar z) + \bar z' \Lambda_t \bar z + Y_t' z + R_t,
  \end{equation}
   where the dynamics of the coefficients $K,\Lambda, Y,R$ are given by
   \beqs
   dK_t \; = \; \dot K_t dt, \qquad
   d\Lambda_t \; = \; \dot \Lambda_t dt, \qquad
   dY_t \; = \; \dot Y_t dt + Z_t^Y dW^0_t, \qquad
   dR_t \; = \; \dot R_t dt + Z_t^R dW^0_t,
   \enqs
for some symmetric deterministic processes $\dot K, \dot \Lambda$ and $\F^0$-adapted processes $\dot Y$, $\dot R$, $Z^Y$, $Z^R$. As in Appendix \ref{AppA1}, we have assumed the quadratic coefficient to be deterministic. Moreover, since the randomness only comes from the $\F^0$-adapted process $T$, the stochastic part in $Y,R$ only depends on $W^0$.
   
By applying It\^o's formula to $S_t^\delta$ $=$ $e^{-\rho t} v_t(Z_t^\delta, \eee[Z_t^\delta])$ $+$ $\int_0^t e^{-\rho s} \tilde g_s(Z^\delta_s, \eee[Z^\delta_s], \delta_s)ds$ and completing the square, we get $d\E[S_t^\delta] = e^{-\rho t} \E[ \Dc_t^\delta ] dt$, with
\beq \label{SOCPLito}
   \E[\Dc_t^\delta]  &=& \E \Big[ \big( \delta_t  - A(Z_t^\delta,\E[Z_t^\delta],K_t,\Lambda_t,Y_t) \big)' N \big( \delta_t  - A(Z_t^\delta,\E[Z_t^\delta],K_t,\Lambda_t,Y_t) \big)  \\
   & & \;\;\;\;\; + \;  \big(Z_t^\delta  - \E[Z_t^\delta] \big)' F(K_t, \dot K_t) \big(Z_t^\delta - \E[Z_t^\delta] \big) \\ & & \;\;\;\;\; + \; \E[Z_t^\delta]' G(K_t, \Lambda_t, \dot \Lambda_t) \E[Z_t^\delta] \nonumber\\
   & & \;\;\;\;\; + \; H_t(K_t, \Lambda_t, Y_t,\eee[Y_t],\dot Y_t)' Z_t^\delta + \; M(Y_t, R_t, \dot R_t)  \Big], \nonumber
 \enq
where the coefficients are defined by
\beqs
F(k, \dot k) &=& \dot k - kN^{-1}k + SkS - \rho k + Q,   \\
G(k,\lambda, \dot \lambda) &=& \dot \lambda -\lambda N^{-1} \lambda - \rho \lambda + SkS + Q + \tilde Q,  \\
H_t(k, \lambda, y, \bar y, \dot y) &=& \dot y - \rho y - k N^{-1}(y - \bar y) - \lambda N^{-1} (U + \bar y) + T_t,  \\
M(y, r, \dot r) &=& \dot r - \rho r  - \frac{1}{4} (U + y)'N^{-1}(U+y),  \\
A(z, \bar z, k, \lambda, y) &=& - N^{-1}k(z - \bar z) - N^{-1} \lambda \bar z - N^{-1}(U+y)/2.
\enqs
Setting $\E[\mathcal{D}^\delta_t] \geq 0$ for each $\delta$ and $\E[\mathcal{D}^\delta_t] = 0$ for $\delta = \delta^*$ provides the optimal control
\beq  \label{SOCPLimplicit}
  \delta^*_t &=& A(Z_t^{\delta^*},\E[Z_t^{\delta^*}],K_t,\Lambda_t,Y_t)
\enq
and the following equations for the coefficients
\beqs
dK_t &=& \Big( K_tN^{-1}K_t - SK_tS  +\rho K_t - Q\Big) dt, \\
d\Lambda_t &=& \Big(\Lambda_tN^{-1}\Lambda_t + \rho \Lambda_t - SK_tS - Q - \tilde Q \Big) dt, \\
dY_t &=& \Big(\rho Y_t  + K_t N^{-1} (Y_t - \eee[Y_t]) + \Lambda_t N^{-1}(U + \eee[Y_t]) - T_t \Big) dt  + Z^Y_t dW^0_t,\\
dR_t &=& \Big(\rho R_t + \frac{1}{4}(U + Y_t)'N^{-1}(U + Y_t) \Big) dt   + Z^R_t dW^0_t.
\enqs
The first and the second equation admit constant solutions $K_t \equiv K$ and $\Lambda_t \equiv \Lambda$: there exists a unique couple of symmetric positive-definite matrices $(K,\Lambda)$ which solves the equations
\begin{gather}
\label{eqK}
KN^{-1}K - SKS  +\rho K - Q = 0,
\\
\label{eqL}
\Lambda N^{-1}\Lambda + \rho \Lambda - SKS - Q - \tilde Q =0,
\end{gather}
i.e.~the systems \eqref{SOCPLdefKL} (we postpone the proof to the end of this appendix). The third and the fourth equations are linear BSDEs, with immediate solution 
(for the coefficient $Y_t$, first compute $Y_t - \eee[Y_t]$):
\beq  
  Y_t &=& \int_t^\infty e^{-(\rho \,\text{Id} + K N^{-1})(s-t)}  \eee[ T_s| \mathcal{F}^0_t] ds   \label{SOCPLdefYR} \\
 & & + \int_t^\infty \Big( e^{-(\rho \,\text{Id} + \Lambda N^{-1})(s-t)} - e^{-(\rho \,\text{Id} + K N^{-1})(s-t)} \Big) \eee[T_s] ds + (\rho \,\text{Id} + \Lambda N^{-1})^{-1}(\rho U) - U.  \nonumber \\
 R_t &=& - \frac{1}{4} \int_t^\infty e^{-\rho (s-t)} \eee[(U + Y_s)'N^{-1}(U + Y_s)| \mathcal{F}^0_t] ds. \nonumber
 \enq
Since $K,\Lambda>0$, the matrices $\rho \,\text{Id} + K N^{-1}$ and $\rho \,\text{Id} + \Lambda N^{-1}$ are positive-definite and invertible; then, by arguing as in Appendix \ref{AppA1}, the processes $Y_t$ and $R_t$ are well-defined. We get \eqref{SOCPLoptimal} by plugging \eqref{SOCPLdefYR} into \eqref{SOCPLimplicit} and setting $\tilde Y = Y+U$, whereas \eqref{SOCPLproduction} immediately follows by \eqref{dynZ} and \eqref{SOCPLoptimal}.  Moreover, the conditions (i) and (iii) in Lemma \ref{lemverif} are verified by standard estimates as in Appendix \ref{AppA1}. Further, we get the limit result by the same computations as the ones in Proposition \ref{prop:limitscons}. \ep  

\begin{Lemma}
There exists a unique couple of symmetric positive-definite matrices $(K,\Lambda)$ which solves the equations \eqref{eqK}-\eqref{eqL}.
\end{Lemma}
  
\begin{proof}
Uniqueness immediately follows from \eqref{valfunc}. To prove the existence of the solutions, we  link \eqref{eqK}-\eqref{eqL} to suitable control problems. For $T \in \rr^+ \cup \{\infty\}$ and $x \in \rr^2$, we consider
\begin{gather*}
V_T(x)= \inf_{u \in \mathcal{U}_T} \eee \left[ \int_0^T e^{-\rho s} (X'_s Q X_s + u'_s N^{-1} u_s) ds\right],
\\
\mathcal{U}_T = \Big\{ \text{$\rr^2$-valued adapted $u = \{u_s\}_{s \in [0,T]}$ s.t. } \eee \Big[ \int_0^T e^{-\rho s} |u_s|^2 ds \Big] < \infty \Big\},
\\
dX_s = u_s ds + S X_s dW_s, \qquad X_0=x.
\end{gather*}
By arguing as in \eqref{integX}, it is easy to see that $u \in \mathcal{U}_T$ implies $\eee \big[ \int_0^T e^{-\rho s} |X_s|^2 ds \big]$ $<$ $\infty$, so that the problems are well-defined. If $T$ is finite, we know (see \cite[Sections 6.6 and 6.7]{YongZhou}, with a straightforward adaptation of the arguments to include the discount factor) that there exists a unique solution $\{K_{t;T}\}_{t \in [0,T]}$ to
\begin{equation}
\label{ODEapprox}
\begin{cases}
\frac{d}{dt} K_{t;T} = \rho K_{t;T} - SK_{t;T}S - Q + K_{t;T} N^{-1} K_{t;T},
\\
K_{T;T}=0,
\end{cases}
\end{equation}
and that for every $x \in \rr^2$ we have
\begin{equation*}
V_T(x) = x'K_{0;T}x.
\end{equation*}
It is easy to see that $V_T \to V_\infty$ as $T \to \infty$; as a consequence, there exists 
\begin{equation*}
\lim_{T \to \infty} K_{0;T} =: \tilde K.
\end{equation*}
By a classical argument and since the functions $\{K_{t;T}\}_{t \in [0,T]}$ solve \eqref{ODEapprox}, $\tilde K$ is a solution to \eqref{eqK}. Also, $\tilde K$ is symmetric as it is the limit of symmetric matrices. Moreover, notice that $N^{-1}>0$ and assume that $Q>0$; by standard arguments in control theory, it follows that  $x'\tilde K x = V_\infty(x)>0$ for each $x \neq 0$, so that $\tilde K>0$. Hence, we have proved that \eqref{eqK} admits a symmetric positive-definite solution provided that $Q>0$, which is true by \eqref{CondParamSocPl}.

By similar arguments, \eqref{eqL} admits a symmetric positive-definite solution provided that $SKS + Q + \tilde Q>0$, which is true by \eqref{CondParamSocPl}, since $K^{11}>0$.
\end{proof}

\vspace{9mm}
\small
\bibliographystyle{plain}

\begin{thebibliography}{1}
	
\bibitem{Acemoglu16}
D.~Acemoglu, U.~Akcigit, D.~Hanley, W.~Kerr.
\newblock Transition to clean technology.
\newblock {\em Journal of Political Economy}, 124(1):52--104, 2016.

\bibitem{Acemoglu12}
D.~Acemoglu, P.~Aghion, L.~Bursztyn, D.~Hemous.
\newblock The Environment and Directed Technical Change.
\newblock {\em The American Economic Review}, 102(1):131--166, 2012.

\bibitem{baspha17}
M.~Basei, H.~Pham.
\newblock  A weak martingale approach to linear-quadratic McKean-Vlasov stochastic control problems, 
\newblock {\it Journal of Optimization Theory and Applications}, to appear 2019.
 
\bibitem{benetal13} 
A.~Bensoussan, J.~Frehse, P.~Yam. 
\emph{Mean Field Games and Mean Field Type Control Theory},
Springer Briefs in Mathematics, 2013. 

\bibitem{Boiteux60}
M.~Boiteux.
\newblock Peak-load pricing.
\newblock {\em  The Journal of Business}, 33(2):157--179, 1960.

\bibitem{Boiteux56}
M.~Boiteux.
\newblock Sur la gestion des Monopoles Publics astreints à l'équilibre budgétaire.
\newblock {\em  Econometrica}, 24(1):22--40, 1956.

\bibitem{Borenstein12}
S.~Borenstein.
\newblock The Private and Public Economies of Renewable Electricity Generation.
\newblock {\em The Journal of Economic Perspectives}, 26(1):67--92, 2012.
 
\bibitem{Brown17}
\newblock D. P. Brown, D. Sappington.
\newblock Designing compensation for Distributed solar generation: Is net metering ever optimal?
\newblock {\em Energy Journal}, 38(3), 2017.
 
\bibitem{Bunde18}
\newblock Bundesnetzagentur für Elektrizität, Gas, Telekomunikation, Post und Eisenbahnen.
\newblock Monitoring Report, 2018.

 
\bibitem{cardel17} 
R.~Carmona, F.~Delarue. 
\newblock {\it Probabilistic theory of mean field games with applications I-II}, Springer Verlag, 2018. 

 
\bibitem{Gowrisankaran16}
G.~Gowrisankaran, S.~S.~Reynolds, M.~Samano.
\newblock Intermittency and the Value of Renewable Energy.
\newblock {\em Journal of Political Economy}, 124(4):1187--1234, 2016.

\bibitem{Hirth13}
L.~Hirth.
\newblock The market value of variable renewables. The effect of solar wind power variability on their relative price.
\newblock {\em Energy Economics}, 38:218--236, 2013.

\bibitem{Hirth16}
L.~Hirth.
\newblock The Optimal Share of Variable Renewables: How the Variability of Wind and Solar Power affects their Welfare-optimal Deployment.
\newblock {\em The Energy Journal}, 36(1):149--184, 2013.

\bibitem{Hirth15}
L.~Hirth, F.~Ueckerdt, O.~Edenhofer.
\newblock Integration costs revisited - An economic framework for wind and solar variability.
\newblock {\em Renewable Energy}, 74:925--939, 2015.

\bibitem{Joskow11}
P.~Joskow.
\newblock Comparing the Costs of Intermittent and Dispatchable Electricity Generating Technologies.
\newblock {\em The American Economic Review}, 101(3):238--241, 2011.
 
\bibitem{elk81} 
N.~El Karoui. 
\newblock Les aspects probabilistes du contr\^ole stochastique, Lect. Notes in Math., 816, Springer Verlag, 1981.
 
\bibitem{Green15}
R.~Green, T.-O.~Léautier.
\newblock Do costs fall faster than revenues? Dynamics of renewables entry into electricity markets.
\newblock {\em TSE Working Paper}, wp-TSE-591, 2015.
 
\bibitem{phawei16} 
H.~Pham, X.~Wei. 
\newblock Dynamic programming for optimal control of stochastic McKean-Vlasov dynamics, 
\newblock{\em SIAM Journal on Control and Optimization}, 55(2):1069--1101, 2017. 

\bibitem{Reichelstein13}
S.~Reichelstein, M.~Yorston.
\newblock The prospects for cost competitive solar PV power.
\newblock {\em Energy Policy}, 55:117--127, 2013.

\bibitem{Renaud93}
A. Renaud.
\newblock Daily generation management at Electricité de France: From Planning towards Real--Time.
\newblock {\em IEEE Trans. on Automatic Control}, 38(7):1080--1093, 1993.

\bibitem{Rubin15}
E.~S.~Rubin, I.~M.~L.~Azevedo, P.~Jaramillo, S.~Yeh.
\newblock A review of learning rates for electricity supply technologies.
\newblock {\em Energy Policy}, 86:198--218, 2015.

\bibitem{Wiser11}
R. Wiser, M. Bollinger.
\newblock 2010 Wind technologies market report.
\newblock {\em US Department of energy}, 2011.


\bibitem{YongZhou}  
J.~Yong, X.~Y.~Zhou. 
\newblock {\it Stochastic controls}, Springer Verlag, 1999.
\end{thebibliography}

\end{document}